\documentclass{tran-l}
\usepackage[square,sort,comma,numbers]{natbib}
\setcitestyle{square}
\usepackage{amssymb}
\usepackage{mathtools}
\usepackage{mathrsfs}
\usepackage{chemfig}
\usepackage{chemformula}
\usepackage{centernot}

\usepackage{url}
\usepackage{hyperref}
\hypersetup{
    linktocpage,
    colorlinks,
    citecolor=black,
    filecolor=black,
    linkcolor=blue,
    urlcolor=black
}
\usepackage{xcolor}
\usepackage[left=3cm, right=3cm]{geometry}
\usepackage{braket}
\usepackage{physics}
\usepackage{amscd}
\usepackage{tikz-cd} 
\usepackage{tikz}
\usetikzlibrary{shapes.arrows}
\usetikzlibrary{positioning}
\usepackage{calc}
\usepackage{float}
\tikzset{%
  symbol/.style={
    draw=none,
    every to/.append style={
      edge node={node [sloped, allow upside down, auto=false]{$#1$}}
    },
  },
}

\usepackage{caption}
\newtheorem{theorem}{Theorem}

\newtheorem{example}{Example}
\newtheorem{question}{Question}
\newtheorem{proposition}[theorem]{Proposition}
\newtheorem{corollary}[theorem]{Corollary}
\newtheorem*{summary*}{Summary}

\theoremstyle{definition}
\newtheorem{definition}{Definition}
\newtheorem*{"definition"}{"Definition"}

\theoremstyle{remark}
\newtheorem{remark}{Remark}

\numberwithin{equation}{section}

\usepackage[utf8]{inputenc}
\usepackage{bxcjkjatype}
\begin{document}

\title{Quantum Fibrations: quantum computation on an arbitrary topological space}


%    author one information
\author{Kazuki Ikeda}
\address{}
\curraddr{}
\email{kazuki7131@gmail.com}
\address{Department of Mathematics and Statistics $\&$ Centre for Quantum Topology and Its Applications (quanTA), University of Saskatchewan, Canada}
\address{Co-design Center for Quantum Advantage (C2QA) $\&$ Center For Nuclear Theory, Department of Physics and Astronomy, Stony Brook University, USA}

%\subjclass is required.
%\subjclass[2010]{Primary}

\date{}

\dedicatory{}

%    Abstract is required.
\begin{abstract}
Using von Neumann algebras, we extend the theory of quantum computation on a graph to a theory of computation on an arbitrary topological space.
\end{abstract}

\maketitle
\setcounter{tocdepth}{1}
\tableofcontents
\section{Introduction}
\subsection{Quantum Computation}
Throughout this article, any Hilbert space we consider is separable and complex. All operators on any Hilbert space are assumed to be trace class and bounded. Let $B(\mathcal{H})$ be the set of all bounded operators on a Hilbert space $\mathcal{H}$. For $A\in B(\mathcal{H})$, we write its conjugate as $A^*$. Let $D(\mathcal{H})=\{\rho\in B(\mathcal{H}):\Tr\rho=1,\rho^*=\rho,\rho\ge0\}$ be the set of all density operators (quantum states) acting on $\mathcal{H}$. Let $\rho$ be an endomorphism of $\mathbb{C}^{2^n}$ such that $\Tr\rho=1$, $\rho^*=\rho$ and $\rho\ge0$. Such an operator $\rho$ is called a quantum state or density operator. A quantum state is called pure if $\Tr(\rho^2)=1$; otherwise, it is called mixed. Let $\{E^k_{j}\}_{k=1,j=1}^{K,J_k}$ be a set of endomorphisms (Kraus operators) of $\mathbb{C}^{2^n}$ such that
\begin{equation}
\label{eq:Kraus}
    \sum_{k=1}^K\sum_{j=1}^{J_k}(E^k_j)^* E_j^k=I_{2^n}.
\end{equation}
The measurement of the operator $\rho$ is defined using $\{E^k_{j}\}_{k=1,j=1}^{K,J_k}$ and the result $k$ can be obtained with probability
\begin{equation}
\Tr\left(\sum_{j=1}^{J_k}E^k_j\rho(E^k_j)^*\right).
\end{equation}
If the measurement result is $k$, the state $\rho$ is transformed into another state
\begin{equation}
\rho'=\frac{\sum_{j=1}^{J_k}E^k_j\rho(E^k_j)^*}{\Tr\left(\sum_{j=1}^{J_k}E^k_j\rho(E^k_j)^*\right)}.
\end{equation}
Quantum computation $\left(\rho_{in},\{U_{\lambda}\}_{\lambda\in\Lambda},\{E^k_{j}\}_{k=1,j=1}^{K,J_k}\right)$ is defined by an initial quantum state $\rho_{in}$, a family of unitary operators $\{U_\lambda\}_{\lambda\in\Lambda}$ and Kraus operators $\{E^k_{j}\}_{k=1,j=1}^{K,J_k}$. The initial state $\rho_{in}$ which is updated by sequentially applying unitary operators $\{U_\lambda\}_{\lambda\in\Lambda}$ as $\rho_{in}\mapsto U_\lambda\rho_{in}U^*_\lambda$. Quantum computation is called universal if any $2^n$-dimensional unitary operator can be approximated using several operators from $\{U_\lambda\}_{\lambda\in\Lambda}$.

The physical system that performs quantum computation is realized by graphically arranging $n$ artificial atoms, called qubits. Each qubit is an unit element of $\mathbb{C}^2$. Although theoretical methods for realizing universal quantum computation are well established, it is also well known that it is very difficult to simulate arbitrary physical processes by lattice computation. To explain this issue more precisely, let us illustrate the general procedure of quantum computation.

Let $(\Omega,\mathscr{F},P)$ be an arbitrary probability space where a given problem is defined. It can be solved by quantum computation in the following way:
\begin{itemize}
    \item[1.] Assign $\mathbb{C}^{2^{n}}$ and $\left(\rho_{in},\{U_{\lambda}\}_{\lambda\in\Lambda},\{E^k_{j}\}_{k=1,j=1}^{K,J_k}\right)$, where each $U_\lambda$ is generated by $n$-qubit Pauli operators. 
    \item[2.] For a given $\epsilon>0$, find an embedding map $f:\Omega\to \left(\rho_{in},\{U_{\lambda}\}_{\lambda\in\Lambda},\{E^k_{j}\}_{k=1,j=1}^{K,J_k}\right)$ such that for all $\omega\in\Omega$, there exists $i,\mu$ that satisfy $\left|P(\omega)-\Tr(\sum_{j=1}^{J_i}E^i_jU_\mu\rho_{in}U^*_\mu(E^i_j)^*)\right|<\epsilon$. 
\end{itemize}

The technical difficulties in this process are as follows:
\begin{itemize}
    \item[(A)] A sufficient number of qubits, which give $\mathbb{C}^{2^n}$, and operators $\left(\rho_{in},\{U_{\lambda}\}_{\lambda\in\Lambda},\{E^k_{j}\}_{k=1,j=1}^{K,J_k}\right)$ should be given to embed a problem and achieve a given precision $\epsilon$. 
    \item[(B)] An embedding map $f$ should be found so the original distribution $P(\omega)$ can be approximated by $n$-qubit Pauli operators within a given precision $\epsilon$. 
\end{itemize}
The former problem (A) is an experimental difficulty, while the latter (B) involves a theoretical difficulty. For example, to consider the Standard Model of elementary particles from a view point of lattice gauge theory, it is necessary to solve (circumvent) several "No-go theorems", such as the Nielsen-Ninomiya theorem, which states that right handed and left handed quarks and leptons appear in pairs, so chiral symmetry is not realized on a lattice. The construction of lattice fermions with exact chiral symmetry was circumvented by Neuberger's proposal of overlap fermions, but the problem of huge computational cost for simulations appeared. Quantum computation is expected to help reduce such computational costs, but it is generally very difficult to equivalently replace a problem on a space having cardinality of the continuum $2^{\aleph_0}$ (or greater cardinality) with a problem on a graph/lattice having cardinarity of $\aleph_0$ or less. When we say that quantum computation is universal, we mean that it can produce arbitrary probability distributions or unitary operators with arbitrary precision (if we are allowed to use sufficiently long time and a large amount of memory space), but nothing is told about how this is possible. 

\subsection{Statement of Main Results}
One of the finest aspects of quantum computation is that, in principle, it can approximate arbitrary quantum many-body systems defined on any topological space in a well-defined manner. Quantum mechanics is generally described by operators acting in an infinite-dimensional complex Hilbert space, but in quantum computation they can be approximated by a finite number of operators acting in a finite-dimensional complex linear space. This gives an advantage when simulating quantum field theory and quantum gravity. However, embedding quantum field theory involves technical difficulties, as mentioned earlier. To solve these problems, it is natural to consider a model of quantum computation based on the quantum theory of fields. However, in~\cite{freedman2002simulation}, it is reported that topological quantum field theory (TQFT) cannot be used to define a model of computation stronger than BQP. The above state of affairs prompts the following question.

\begin{question}
\label{thm:q0}
    How to define a computational theory that is more powerful than the conventional quantum computational model?
\end{question}

The most general framework that can address any quantum theory will be given as follows.

\vspace{5mm}
{\setlength{\leftskip}{\leftskip+5truemm}\setlength{\rightskip}{\rightskip+5truemm}%
\textit{Let $\mathscr{F}$ and $X$ be topological spaces, and $(\mathscr{F},\pi,X)$ be a triple that satisfies the following conditions: 
\begin{enumerate}
\item $\pi:\mathscr{F}\to X$ is continuous.
\item For each open cover $\{U_\lambda\}_{\lambda\in\Lambda}$ of $X$, $\pi^{-1}(U_\lambda)$ is a set of quantum states for every $\lambda\in \Lambda$.
\end{enumerate}
}\par}% グループ内に \par が必要な点に注意！
\vspace{5mm}

First of all, in the following way, one can check this framework is the most general one that can address any quantum theory. As we see below, continuity of $\pi:\mathscr{F}\to X$ is not a strong condition. Given a topological space $(X,\mathcal{O}(X))$, for any set $\mathscr{F}$ and any map $\pi:\mathscr{F}\to X$, we can define a topology into $\mathscr{F}$ so that $\pi$ is a continuous map. Let $\mathcal{O}_\pi(\mathscr{F})=\{U\subset \mathscr{F}:\exists V\in\mathcal{O}(X),~U=\pi^{-1}(V)\}$ be a family of subsets of $\mathscr{F}$. Then $(\mathscr{F},\mathcal{O}_\pi(\mathscr{F}))$ is a topological space and $\pi:\mathscr{F}\to X$ is a continuous map. For the second condition, let $\mathscr{F}$ be the set of all quantum states on $X$ and $\pi:\mathscr{F}\to X$ be a surjection. Then it is possible to simulate/approximate any quantum state of any quantum theory on $X$. 

Given such a triple $(\mathscr{F},\pi,X)$, one can define and run a problem by applying algebraic operators to each $\pi^{-1}(U)$. Let us consider how we can program each $\pi^{-1}(U)$ to solve a problem. To do so, we first decide what kind of quantum state we will address. Since all experimentally observed physical quantities are of finite value, it is sufficient to consider bounded operators.
Therefore we use a von Neumann algebra ($W^*$-algebra) to embed the problem into $\mathscr{F}$ and write a program. Here a von Neumann algebra is a weakly closed $^*$-algebra of bounded operators on a Hilbert space and contains the identity operator. Note that quantum mechanics includes not only bounded operators but also non-bounded operators. Let $\mathcal{A}(U)$ be a von Neumann algebra on an open set $U$ of $X$. It will act on $\pi^{-1}(U)$ as follows. Suppose a initial quantum state $\rho_U\in\mathcal{A}(U)$ is given on $U$. Then one can update it by applying a family of unitary operators $\{a_\lambda\}_\lambda\subset\mathcal{A}(U)$ to $\rho_U$ in such a way that $\rho_U\mapsto a_\lambda \rho_U a^*_U$. One will be able to simulate a local behavior on an open subset $V\subset U$, by restricting $\mathcal{A}(U)$ to $V$, by which one will obtain a subalgebra $\mathcal{A}(V)\subset\mathcal{A}(U)$. In order to extend $\rho_U$ to a state on $W\supset U$, one could do so by embedding $\mathcal{A}(U)$ to $\mathcal{A}(W)$. This extension includes the notion of a tensor product of quantum states. If we can give such an algebraic system to an arbitrary open set of $X$ and reproduce arbitrary quantum states, we can call it a true universal quantum computation. When those operations are defined on a triple $(\mathscr{F},\pi,X)$, we call it \textbf{quantum fibration} (Definition~\ref{def:fibration}). This is clearly a natural extension of conventional quantum computation, which discretely embeds a problem into each fiber $\mathbb{C}^{2^n}$ using $n$-qubit Pauli operators.

Another motivation for the author in developing the theory is as follows:

\vspace{5mm}
{\setlength{\leftskip}{\leftskip+5truemm}\setlength{\rightskip}{\rightskip+5truemm}%
\textit{This theory is a well-defined and non-perturbative quantum theory on any topological space.}\par}% グループ内に \par が必要な点に注意！
\vspace{5mm}

Conventional quantum theory tries to find a Hamiltonian or Lagrangian a priori to solve a problem, whereas quantum computation finds an algebraic system that can be programmed to reproduce a quantum system to solve a problem. Moreover algebraic operations are well-defined, and if the quantum computation is universal, it can reproduce any state, regardless of non-perturbative/perturbative.

\vspace{5mm}
Since quantum computation involves classical computation, we come to the following natural question:
\begin{question}
\label{thm:q0-1}
How to address classical information with a quantum fibration? 
\end{question}
The answer to the above question is given in Definition~\ref{def:semiclassical}. We propose the notion of \textbf{semi-classical operations of von Neumann algebra}. By restricting each fiber to semi-classical orbits, we would be able to handle semi-classical systems. Quantum and classical states can be distinguished by the presence or absence of quantum correlations such as entanglement. Density operators without entanglement are called separable states and are regarded as classical states. The strength of the entanglement can be measured with what is called an entanglement measure (cf. Example~\ref{ex:EEmeasure}.) For example we can define a semi-classical operator by a von Neumann algebra such that it does not increase quantum correlation. In Theorem~\ref{def:semiclassical} we verify that any semi-classical von Neumann algebra is indeed closed in classical system. 

As some examples of our theory, we address algebraic quantum field theory in Section~\ref{sec:AQFT} and quantum chemistry in Section~\ref{sec:chemi}.

\vspace{5mm}
%\section*{acknowledgement}
\hspace{-4mm}\textit{Acknowledgements.} I gratefully acknowledge people at Stony Brook University especially Dmitri Kharzeev, Edward Shuryak, Derek Teaney, Raju Venugopalan, Jacobus Verbaarschot, Ismail Zahed for helpful and stimulating discussions. I thank Pablo Basteiro, Ioannis Matthaiakakis, Rene Meyer at University of W\"{u}erzburg, Adam Lowe and Yoshiyuki Matsuki for useful discussions and collaboration. I also thank Steven Rayan for his encouragement and supervision at University of Saskatchewan. This work was supported in part by Pacific Institute for the Mathematical Science (PIMS) postdoctoral fellowship award and by the U.S. Department of Energy, Office of Science, National Quantum Information Science Research Centers, Co-design Center for Quantum Advantage (C2QA).

\section{Orbits and Quantum Operations}
\label{sec:orbit}
Relaxing the condition of density operators, we define the 
following of positive Hermitian operators as
\begin{equation}
    \widetilde{D}(\mathcal{H})=\{\rho\in B(\mathcal{H}):\rho^*=\rho,~\rho\ge0,\Tr\rho>0\}. 
\end{equation}
Any element $\rho$ of $\widetilde{D}(\mathcal{H})$ can be identified with a density operator by dividing by the trace $\rho/\Tr\rho$. So $\widetilde{D}(\mathcal{H})$ is the set of all quantum states before normalizing.  

Let $\mathcal{A}$ be a $W^*$-algebra (von Neumann algebra), which is a weakly closed $^*$-algebra of bounded operators on a Hilbert space $\mathcal{H}$ and contains the identity operator. Equivalently, a set $\mathcal{A}\subset B(\mathcal{H})$ is called a von Neumann algebra if $(\mathcal{A}')'=\mathcal{A}$, where
\begin{equation}
\label{eq:vonNeumann}
    \mathcal{A}'=\{a\in B(\mathcal{H}): ab-ba=0~\forall b\in\mathcal{A}\}.
\end{equation}

We define a subset of $\mathcal{A}$ by
\begin{equation}
\label{eq:subalgebra}
    \check{\mathcal{A}}=\left\{a\in \mathcal{A}: \forall\rho\in \widetilde{D}(\mathcal{H}),~a\rho a^*\in\widetilde{D}(\mathcal{H})\right\}.
\end{equation}
Moreover let $\mathcal{R}=\{a\in\mathcal{A}: \exists a^{-1}\in\mathcal{A},a^{-1}a=aa^{-1}=e\}$ denote the set of all regular elements of $\mathcal{A}$. 

\begin{proposition}
\label{eq:selfadjoint}
  For any von Neumann algebra $\mathcal{A}$, $\check{\mathcal{A}}$ \eqref{eq:subalgebra} is a self-adjoint set. $\check{\mathcal{A}}\cap\mathcal{R}$ is a group. 
\end{proposition}
\begin{proof}
This can be confirmed by checking that, for any $a\in\check{\mathcal{A}}$ and for any $\rho\in\widetilde{\mathcal{D}}(\mathcal{H})$, $a^*$ also satisfies $i_{a^*}(\rho)\in\widetilde{\mathcal{D}}(\mathcal{H})$. Hence any $a^*$ is an element of $\check{\mathcal{A}}$. Therefore $\check{\mathcal{A}}={\check{\mathcal{A}}}^*=\{a^*,a\in\check{\mathcal{A}}\}$, which means $\check{\mathcal{A}}$ is self-adjoint by definition.

Obviously $\check{\mathcal{A}}$ contains the identity operator $e$, which sends any element of $\widetilde{\mathcal{D}}(\mathcal{H})$ to itself $(e:\rho\mapsto \rho)$. Then for any $a\in\check{\mathcal{A}}\cap\mathcal{R}$ and any $\rho\in\widetilde{\mathcal{D}}(\mathcal{H})$, $a\rho a^*$ is in $\widetilde{\mathcal{D}}(\mathcal{H})$ and $\check{\mathcal{A}}\cap\mathcal{R}$ is closed by the operation $a\rho a^*\mapsto \rho$. Therfore $a^{-1}$ is also an emelemt of $\check{\mathcal{A}}\cap\mathcal{R}$. 

We define the action (quantum operation) by $i_a(\rho)=a\rho a^*$ 
\begin{equation} \left(\check{\mathcal{A}},\widetilde{D}\right)\ni(a,\rho)\mapsto i_a(\rho)\in\widetilde{D}.
\end{equation}
It is clear that the following properties are satisfied:
\begin{align}
    \begin{aligned}
        i_{ab}&=i_ai_b,~\forall a,b\in\check{\mathcal{A}}\\
        i_e&=\text{Id}_{\tilde{D}}\\
        i_ai_{a^{-1}}&=i_{a^{-1}}i_a=\text{Id}_{\tilde{D}},~\forall a\in\check{\mathcal{A}}\cap\mathcal{R}
    \end{aligned}
\end{align}
Hence $\check{\mathcal{A}}\cap\mathcal{R}$ is a group. 
\end{proof}

Moreover it is also clear that 
\begin{equation}
i_a(\rho+\sigma)=a(\rho+\sigma)a^*=(a\rho a^*)+(a\sigma a^*)=i_a(\rho)+i_a(\sigma). 
\end{equation}

For a given $\rho\in\widetilde{D}(\mathcal{H})$ and $a\in\check{\mathcal{A}}$, we have an orbit 
\begin{equation}
    O_{\check{\mathcal{A}}}(\rho)=\left\{\frac{i_a(\rho)}{\Tr(i_a(\rho))}:a\in \check{\mathcal{A}}\right\}. 
\end{equation}
We call it a quantum orbit with an input $\rho$. Note that any von Neumann algebra on a separable Hilbert space has a countable number of generators.

Quantum computation based on this definition of quantum circuits includes not only gate-based quantum computation, but also adiabatic quantum computation and quantum annealing. A quantum gate computation consists of a countable set of elements of $\check{\mathcal{A}}$. In the case of quantum annealing and adiabatic quantum computation, it is based on a one-parameter family $\{a_t\}_{0\le t\le 1}\subset\check{\mathcal{A}}$. 

\begin{example}
One shall find that $\check{\mathcal{A}}$ includes operators for measurement and communication. In the conventional quantum computation, one can create measurement operator $M_m$ by a Kraus operator~\eqref{eq:Kraus} in such a way that 
\begin{equation}
\label{eq:POVM}
M_m=\sum_{j=1}^{J_m}(E^m_j)^* E_j^m.  
\end{equation}
The measurement operator $M_m$, which is called a positive-operator valued measure (POVM), acts on a normalized quantum state $\rho\in D(\mathcal{H})$ and a result $k$ can be obtained with probability 
\begin{equation}
    \Tr(M_k\rho).
\end{equation}
\end{example}

\begin{definition}
\label{def:quantum_universal}
An algebra $\check{\mathcal{A}}$ defined by \eqref{eq:subalgebra} is called {\textbf{universal}} if it acts transitively on $\mathcal{D}$, namely 
\begin{equation}
    O_{\check{\mathcal{A}}}(\rho)=D(\mathcal{H})
\end{equation}
is satisfied for any $\rho$ in $D(\mathcal{H})$.
\end{definition}

In practice, technical reasons may limit input states that can be prepared. So, let $\mathcal{D}\subset\widetilde{D}(\mathcal{H})$ be a set of states, and write the set of all orbits that can be generated using the given algebra $\check{\mathcal{A}}$ as follows 
\begin{equation}
    \mathcal{Q}(\mathcal{D}, \check{\mathcal{A}},\mathcal{H})=\bigcup_{\rho\in \mathcal{D}}O_{\check{\mathcal{A}}}(\rho).
\end{equation}
We call this a programmed quantum system.

\section{Quantum Fibrations and Quantum Networks}
\subsection{Quantum Fiberations}
\begin{definition}
\label{def:fibration}
Let $\mathscr{F}$ and $X$ be topological spaces and $\pi:\mathscr{F}\to X$ be a continuous map. We call $(\mathscr{F},\pi,X)$ a \textbf{quantum fiber space or quantum fibration} when every non-empty fiber $\mathscr{F}_x=\pi^{-1}(x)$ is a programmed quantum system for every $x\in X$. 
\end{definition}
Here, we will admit the existence of fibers that are empty sets. Those empty fibers correspond to the fibers vanished by annihilation operators of particles. The definition of Fibration usually employs the covering homotopy property, but we do not. There are several reasons for this, for example, the homotopy of the fiber is not always readily apparent in any given quantum system. Moreover if $p:E\to X$ is a Serre fibration on a pathwise connected space $X$, then $p^{-1}(x)$ and $p^{-1}(y)$ are homotopy equivalent for all $x,y\in X$. However, general quantum systems do not always satisfy this property. For example, in the most general two-qubit system, pure states are parametrized by $S^7$ and a mixed states are parametrized by $SU(4)$, but $S^7$ and $SU(4)$ are not homotopy equivalent. General quantum systems with fibers such that a fiber at one point is consists of pure states and a different fiber at another point consists of mixed states cannot be handled by the conventional fibrations.

Now let us consider quantum computation with a quantum fibration which consists of the following data:
\begin{enumerate}
\item For each open set $U\subset X$, let $\mathcal{A}(U)$ be a von Neumann algebra such that $\mathcal{A}(V)\subset \mathcal{A}(U)$ for all open sets $V\subset U$.
\item A density operator $\rho_U$ on $U$ is a trace-class operator of $\mathcal{A}(U)$ is defined in such a way that $\rho_U=\rho^*_U,\rho_U\ge0,\Tr(\rho_U)=1$, where $1$ is the identity operator of $\mathcal{A}(U)$. 
\item There exists a restriction map $r$ such that, for any open set $V\subset U$, $r_{UV}(\rho_U)$ is a density operator on $V$. 
\item Let $\Lambda$ be an ordered set. The time-evolution of $\rho_U$ is given by a family $\{U_\lambda\}_{\lambda\in\Lambda}$ of unitary operator $U_\lambda U^*_\lambda=1$ in such a way that 
\begin{equation}
    \rho_U\mapsto U_\lambda\rho_U U^*_\lambda \mapsto U_{\lambda'}U_\lambda\rho_U U^*_\lambda U^*_{\lambda'}~~\lambda<\lambda'\in\Lambda.
\end{equation}
\end{enumerate} 
The pair $(\rho_U,\{U_\lambda\}_{\lambda})$ generates an orbit 
\begin{equation}
    O_{\check{\mathcal{A}}(U)}=\{i_{U_\lambda}(\rho_U):\lambda\in\Lambda\}.
\end{equation}
There are two ways to define an orbit on each single point $x\in X$. One is simply to chose a von Neumann algebra $\mathcal{A}_x$ and construct a von Neumann algebra $\mathcal{A}(U)$ on an open neighber $U$ of $x$ in such a way that $\mathcal{A}_x$ is a subalgebra of $\mathcal{A}(U)$. In this case, the restriction map is not given a priori but is defined to be consistent with the algebraic structure on each open subset. Another way to construct an orbit on $x\in X$ is to use the direct limit
\begin{equation}
    \mathcal{A}_x=\lim_{U\to x}\mathcal{A}(U). 
\end{equation}
In this case, a restriction map and algebraic structures on each open neighbor of $x$ are given a priori.   

Once we assign an orbit for each open set $U$ of $X$ with an initial state $\rho_U$, we can define a quantum fibration $(\mathscr{F},\pi,X)$ with continuous map
\begin{equation}
    \pi:\mathscr{F}\to X
\end{equation}
such that $\pi^{-1}(U)=O_{\mathcal{A}_U}(\rho_U)$. 

\begin{example}
\label{ex:qubits}
To find a connection with the conventional quantum computation, let us consider a discrete set $X=\{1,2,3\}$ with the discrete topology $\mathcal{O}_X=\{\emptyset,\{1\}, \{2\},\{3\},\{1,2\}, \{1,3\}, \{2,3\},\{1,2,3\}\}$. This is a system of three qubits which are defined on $1$, $2$ and $3$ in $X$. Single qubit unitary operators can be defined on each of $\{1\},\{2\},\{3\}$, two qubit unitary operators can be defined on $\{1,2\},\{1,3\},\{2,3\}$ and three qubit operations are defined on $\{1,2,3\}$. A restriction map can be interpreted as, for example, partial trace and projection. Note that $B\left(\mathbb{C}^{2^n}\right)=M_{2^n}(\mathbb{C})~(n=3)$ is a von Neumann algebra and all unitary operators acting on qubits in $X$ are elements of $B\left(\mathbb{C}^{2^n}\right)$.
\end{example}

To simulate theories on a connected space, it is natural to consider a space in which there are an infinite number of qubits. For example, this would be the case when considering the $N\to\infty$ limit of $\text{SU}(N)$ gauge theory. Let us check that our theory can also address the quantum computation theory with an infinite number of qubits. To this end, let us discuss tensor products of infinite Hilbert spaces. Let $\{\mathcal{H}_i\}_{i=1,2,\cdots}$ be a sequence of Hilbert spaces and $\{e_i\}_{i=1,2,\cdots}$ be a sequence of their unit vectors $(e_i\in\mathcal{H}_i)$. For each $\bigotimes_{i=1}^n\mathcal{H}_i$, the embedding map 
\begin{equation}
\label{eq:embed}
\psi\in\bigotimes_{i=1}^n\mathcal{H}_i\mapsto \psi\otimes e_{n+1}\in\bigotimes_{i=1}^{n+1}\mathcal{H}_i
\end{equation}
is an isometry map. The direct limit of the direct system $\{\bigotimes_{i=1}^n\mathcal{H}_i\}_{i=1,2,\cdots}$ defined in this way is a pre-Hilbert space, whose completion with $\{e_i\}_{i=1,2,\cdots}$ is called the infinite tensor product of $\{\mathcal{H}_i\}_{i=1,2,\cdots}$. We write it as $\bigotimes_{i=1}^\infty(\mathcal{H}_i,e_i)$. When a von Neumann algebra $\mathcal{A}_i$ is given for each Hilbert space $\mathcal{H}_i$, one can embed $\widetilde{\mathcal{A}_n}=\bigotimes_{i=1}^n\mathcal{A}_i$ into $\bigotimes_{i=1}^{n+1}\mathcal{A}_i$ in such a way that
\begin{equation}
    \widetilde{\mathcal{A}_n}\ni A\mapsto A\otimes I\in\widetilde{\mathcal{A}_n}\otimes\mathcal{A}_{n+1}=\widetilde{\mathcal{A}_{n+1}},
\end{equation}
by which $\widetilde{\mathcal{A}_n}$ is regarded as a subalgebra of $\widetilde{\mathcal{A}_{n+1}}$. This operation is commutative with the embedding \eqref{eq:embed}, hence we can embed $\widetilde{\mathcal{A}_n}$ into $B\left(\bigotimes_{i=1}^\infty(\mathcal{H}_i,e_i)\right)$.
The von Neumann algebra on $\bigotimes_{i=1}^\infty(\mathcal{H}_i,e_i)$ determined in this way is called the infinite tensor product 
\begin{equation}
\label{eq:infinite_tensor}
\mathcal{A}(\{e_i\}_{i})=\bigotimes_{i=1}^\infty\mathcal{A}_i    
\end{equation}
of $\{\mathcal{A}_i\}_{i=1,2,\cdots}$ with respect to $\{e_i\}_{i=1,2,\cdots}$. One important remark is that even when each $\mathcal{A}_i$ is a set of all complex square matrices $(\mathcal{A}_i=M_{m_i}(\mathbb{C}))$, there are uncountably many different infinite tensor products of von Neumann algebras and, depending on a choice $\{e_i\}_{i=1,2,\cdots}$, $\mathcal{A}(\{e_i\}_{i})$ can be type I, II, and III.

\subsection{Quantum Network and Interactions among Fibers}
All operations in all quantum systems can be viewed as communication channels if we view the initial state as the input state and the final state as the output state. A network is an extension of the communication channels to the entire system. A map sending a quantum state to anther quantum state is called a \textbf{quantum channel}. Let $\mathcal{H}_1,\mathcal{H}_2$ be two Hilbert spaces and $H(\mathcal{H}_1)=\{\eta\in B(\mathcal{H}_1):\eta^*=\eta\}$ be the set of all hermitian operators. We call $\Lambda:\text{End}(\mathcal{H}_1)\to\text{End}(\mathcal{H}_2)$ a quantum channel if it satisfies the following properties.
\begin{itemize}
    \item $\Lambda(a\eta_1+b\eta_2)=a\Lambda(\eta_1)+b\Lambda(\eta_2)$ for all $\eta_1,\eta_2$ in $H(\mathcal{H}_1)$ and for all real numbers $a,b$.
    \item $\Lambda(\eta)\ge0$ for any $\eta$ in $H(\mathcal{H}_1)$ such that $\eta\ge0$.
    \item $\Tr(\Lambda(\eta))=\Tr\eta$ for any $\eta$ in $H(\mathcal{H}_1)$.
    \item For any Hilbert space $\mathcal{H}_3$ and for any $\tilde{\eta}$ in $H(\mathcal{H}_1\otimes\mathcal{H}_3)$, if $\tilde{\eta}\ge0$, then  $\Lambda\otimes I(\tilde{\eta})\ge0$.
\end{itemize}

To elaborate more on the fourth property of a quantum channel in a general way, let $\mathcal{A}$ be a von Neumann algebra defined on $\mathcal{H}_1$. Then the tensor product of $\mathcal{A}$ and $B(\mathcal{H}_2)$ corresponds to 
\begin{equation}
    \mathcal{A}\otimes B(\mathcal{H}_2)=\{a\in B(\mathcal{H}_1\otimes\mathcal{H}_2):a_{ij}\in\mathcal{A}\},
\end{equation}
where $a_{ij}$ is defined by equation~\eqref{eq:matrix}. From this, it is straightforward to see that $B(\mathcal{H}_1)\otimes B(\mathcal{H}_2)=B(\mathcal{H}_1\otimes\mathcal{H}_2)$. Especially, when $\mathcal{H}_2=\mathbb{C}^n$, we have $B(\mathcal{H}_1\otimes\mathbb{C}^n)=B(\mathcal{H}_1)\otimes M_n(\mathbb{C})$. Therefore $a\in B(\mathcal{H}_1\otimes\mathbb{C}^n)$ is an $n\times n$ matrix $a=(a_{ij})$ whose components $a_{ij}$ are in $B(\mathcal{H}_1)$. Thus the condition $\Lambda(\eta)\ge0~(0\le\eta\in H(\mathcal{H}_1))$ is not enough to guarantee $\Lambda\otimes I(\widetilde{\eta})\ge0~(0\le\widetilde{\eta}\in H(\mathcal{H}_1\otimes\mathcal{H}_2))$. The physical meaning of this condition is that when there are no interactions between $\mathcal{H}_1$ and $\mathcal{H}_2$, then events in system $\mathcal{H}_2$ do not affect system $\mathcal{H}_1$.

Let $\mathscr{F},\mathscr{G}$ be two quantum fibrations on $X$ and $Y$, respectively. Let $\mathcal{O}_X$ be the set of all open sets of $X$. For $U\in\mathcal{O}_X$ and $V\in\mathcal{O}_Y$, let $L(\mathscr{F}_U,\mathscr{F}_V)$ denote the set of all quantum channels from $\mathscr{F}_U $ to $\mathscr{G}_V$. We write 
\begin{equation}
    L(\mathscr{F},\mathscr{G})=\bigcup_{U\in \mathcal{O}_X}\bigcup_{V\in \mathcal{O}_Y}\{L(\mathscr{F}_U,\mathscr{G}_V)\}
\end{equation}
and call it \textbf{quantum network} from $\mathscr{F}$ to $\mathscr{G}$. $L(\mathscr{F},\mathscr{G})\cup L(\mathscr{G},\mathscr{F})$ is the bidirectional quantum network between $\mathscr{F}$ and $\mathscr{G}$. In algebraic quantum field theory (AQFT), the assignment $U\to \mathcal{A}(U)$ for each $U\in\mathcal{O}_X$ is called a net of local von Neumann algebras. Defining channels among local algebras allows them to interact and communicate, by which we mean it a network of local von Neumann algebras.

\section{Comparison with Theory of Quantum Computation}
\subsection{Commonalities}

\paragraph{\textbf{Measurement}}
Measurements can be defined in the same way as in the traditional theory of quantum computation. Let $\Omega$ be a set and $P_\Omega$ be the set of all subsets of $\Omega$. Let $\mathcal{F}\subset P_\Omega$ be a $\sigma$ algebra, namely it satisfies the following conditions:
\begin{itemize}
    \item $\Omega\in\mathcal{F}$
    \item $B\in\mathcal{F}\Rightarrow B^c=\{\omega\in\Omega:\omega\notin B\}\in\mathcal{F}$
    \item $B_i\in\mathcal{F}~(i=1,2,\cdots)\Rightarrow \bigcup_{i=1}^{+\infty}B_i\in \mathcal{F}$
\end{itemize}

Quantum measurement is generally defined by Positive-Operator Valued Measure (POVM). Again, let $\Omega$ be a set and $P_\Omega$ be a $\sigma$-algebra. We call a map $M:P_\Omega\to B(\mathcal{H})$ is a POVM if the following conditions are satisfied:
\begin{itemize}
    \item For all $B\in P_\Omega$, $M(B)$ is a trace-class self-adjoint positive semidefinite operator $(M(E)=M(E)^*\ge 0)$.
    \item $M(\Omega)=\text{id}_\mathcal{H}$
    \item $\Omega(\emptyset)=0$
    \item For all $B_i,B_j\in P_\Omega$ such that $B_i\cap B_j=\emptyset$, $M(\bigcup_{i}B_i)=\sum_iM(B_i)$. 
\end{itemize}
When a state (density operator) $\rho\in D(\mathcal{H})$ is measured with a POVM $M$, the probability that the measured value is contained in $B$ is given by $\Tr(M(B)\rho)$.

\vspace{3mm}
\paragraph{\textbf{Tensor Products of Quantum Gates}}
Constructing tensor products of multiple qubits is an important way to perform non-local quantum computation as well as to address entanglement. Let us confirm that our theory can consider such tensor products in the same way as before. Let $\mathcal{H}_1$ and $\mathcal{H}_2$ be two Hilbert spaces and $\mathcal{H}_1\otimes\mathcal{H}_2$ be the tensor product. For any $a_i\in B(\mathcal{H}_i)$, one can define $a_1\otimes a_2\in B(\mathcal{H}_1\otimes\mathcal{H}_2)$ uniquely by
\begin{equation}
    (a_1\otimes a_2)(x_1\otimes x_2)=a_1x_1\otimes a_2x_2. 
\end{equation}
With respect to given von Neumann algebras $\mathcal{A}_1\subset B(\mathcal{H}_1)$ and $\mathcal{A}_2\subset B(\mathcal{H}_2)$, their tensor product $\mathcal{A}_1\otimes\mathcal{A}_2$ is defined by the von Neumann algebra generated by 
\begin{equation}
    \{a_1\otimes a_2:a_1\in\mathcal{A}_1,a_2\in\mathcal{A}_2\}.
\end{equation}
While conventional quantum computation theory can only handle finite tensor products, it can be extended to infinite tensor products, as in equation~
\eqref{eq:infinite_tensor}.

Making tensor products allow us to do similar quantum operations in conventional quantum computation and quantum communication. Let $\{\psi_i\}_{i\in I}$ be a complete orthonormal system of $\mathcal{H}_2$. For each $i\in I$, we define an isometry $K_i:\mathcal{H}_1\to\mathcal{H}_1\otimes\mathcal{H}_2$ by
\begin{equation}
\label{eq:op_K}
    K_i\varphi=\varphi\otimes\psi_i
\end{equation}
and $K^*_i:\mathcal{H}_1\otimes\mathcal{H}_2\to\mathcal{H}_1$ by
\begin{equation}
\label{eq:op_Kdag}
K^*_i(\varphi\otimes\psi)=\langle\psi_i,\psi\rangle\varphi~\forall \varphi\in\mathcal{H}_1,\psi\in\mathcal{H}_2.
\end{equation}
Note that $K_iK^*_i$ is the projection from $\mathcal{H}_1\otimes\mathcal{H}_2$ to $\{\varphi\otimes\psi_i:\varphi\in\mathcal{H}_1\}$ and $\sum_{i\in I}K_iK^*_i=I$. Moreover for any $a\in B(\mathcal{H}_1\otimes\mathcal{H}_2)$, we can define a matrix $a=(a_{ij})$ in such a way that
\begin{equation}
\label{eq:matrix}
    a_{ij}=K^*_iaK_j.
\end{equation}
This set $\{K_i\}_{i\in I}$ is essentially a set of Kraus operators~\eqref{eq:Kraus}. 

\vspace{3mm}
\paragraph{\textbf{Digital and Analog Computation}}
When the time evolution of a quantum state is continuous with respect to time, it is called an analog computation; when it is discontinuous, it is called a digital computation. Let $U=\{a: a a^*=1\}\subset\mathcal{A}$ be the set of all unitary operators of von Neumann algebra $\mathcal{A}$. The time evolution of an initial state $\rho_0$ at $t\in \mathbb{R}_{\ge0}$ is defined by 
\begin{equation}
\label{eq:timeevolv}
    \rho_t=a_t\rho_0 a^*_t,~~a_t\in U~a_0=1. 
\end{equation}
Let $I(\ni0)$ be a bounded subset of $\mathbb{R}_{\ge0}$ including 0. We obtain a family $\{\rho_t\}_{t\in I}$ of states. When $I$ is a discrete set, computation is digital. When $I$ is a connected set, computation is analog. Analog quantum computation can be performed by using time-dependent Hamiltonians. Let $H(t)\in\mathcal{A}$ be a Hamiltonian ($H(t)=H^*(t)~\forall t\in I$). When $a_t=\exp(-i\int_0^tH(t)dt)$ is well-defined, then computation with time evolution \eqref{eq:timeevolv} corresponds to adiabatic quantum computation or quantum annealing. 

With respect to digital quantum computation, it is important to know whether it is possible to generate any operator one wishes to perform by means of a countable number of operators in $\mathcal{A}$. The following theorem tells us that digital quantum computation can be universal (see Def.~\ref{def:quantum_universal} for the definition of universal gates). 
\begin{theorem}
Let $\mathcal{H}$ be a separable Hilbert space. Any density operator $\rho\in B(\mathcal{H})$ can be simulated by a countable unitary gate sets $\{a_{t_n}\in\mathcal{A}:a_{t_n}a^*_{t_n}=1\}_{n\in\mathbb{N}}$ of a von Neumann algebra $\mathcal{A}$: 
\begin{equation}
    \exists\rho_0\in D(\mathcal{H}),\forall \rho\in D(\mathcal{H}),\lim_{n\to\infty}\|\rho-a_{t_n}\rho_0a^*_{t_n}\|=0.
\end{equation}
\end{theorem}
\begin{proof}
First of all, as a consequence of the Kaplansky density theorem, it follows that any von Neumann algebra on a separable Hilbert space is generated by a countable set. Therefore with respect to an initial state $\rho_0\in D(\mathcal{H})$ and any target $\rho\in D(\mathcal{H})$, any unitary operator $a_{t_n}$ such that $a^*_{t_n}\rho a_{t_n}=\rho_0$ is generated by discrete time steps $\epsilon_n=t_{n}-t_{n-1}$.
\end{proof}

%For example, every type I von Neumann algebra on a separable Hilbert space has a single generator~\cite{pearcy1962}. 

\subsection{Generalized Things}
\paragraph{\textbf{Base Space: Why are graphs not enough?}}
The use of an arbitrary topological space $X$ for a base space is one of the main extensions from the traditional quantum computation theory defined on graphs. Computations defined only on a graph have various limitations regarding computing power. The class of problems that conventional quantum computers can solve efficiently is called \textbf{BQP}, but it is known that problems defined with an uncountable set generally belong to a higher class than \textbf{BQP}. For example, it is proven that $\textbf{BQP}\subset\textbf{PSPACE}\subsetneq\textbf{NEXPSPACE}$, where
$\textbf{PSPACE}$ is the set of all decision problems problems solvable by a Turing machine using a polynomial amount of space and $\textbf{NEXPSPACE}$ is the set of all decision problems solvable by a non-deterministic Turing machine using an exponential amount of space.  

Therefore even if we use a universal quantum computer, it is difficult to efficiently simulate a generic problem of quantum physics on a space having cardinality of the continuum $2^{\aleph_0}$ (or greater cardinality). The theory of computational complexity $(\textbf{BQP}\subsetneq\textbf{NEXPSPACE})$ suggests that to efficiently simulate general quantum theories with cardinality of the continuum $2^{\aleph_0}$, the memory of a Turing machine should be extended to an exponential amount of space.

In our theory, we are able to address a countably infinite set of quantum gates by considering a infinite tensor product of von Neumann algebras. This would correspond to a case where a set of qubits is dense in an open set of $X$, on which a Turing machine is defined. Moreover when $\{O_\lambda\}_{\lambda\in\Lambda}$ is a family of connected open sets of $X$ such that $X=\bigcup_{\lambda\in \Lambda}O_\lambda$, we can execute quantum computation with cardinality of the continuum $2^{\aleph_0}$ by giving a von Neumann algebra $\mathcal{A}(O_\lambda)$ for each $O_\lambda$.

\vspace{3mm}
\paragraph{\textbf{Quantum Gates and Circuits: From qubits to operator algebras}}
The another important generalization in this work is the use of von Neumann algebras for computation. In the traditional theory, the  Hilbert space is $\mathbb{C}^{2^n}$ with $n$-qubits and operators are elements of $B\left(\mathbb{C}^{2^n}\right)=M_{2^n}(\mathbb{C})$. In this work, we extend the finite dimensional complex Hilbert space $\mathbb{C}^n$ to any separable complex Hilbert space $\mathcal{H}$ and operators are elements of a von Neumann algebra. To see that this is a natural generalization, let us first recall that $M_{2^n}(\mathbb{C})$ is a von Neumann algebra. Each fiber of a quantum fibration accommodates quantum circuits. A Hilbert space $\mathcal{H}$ is called separable if it has a countable orthonormal basis, which is equivalent to $\dim\mathcal{H}\le\aleph_0$.

Another elementary example of a von Neumann algebra is given by multiplication operators. Let $(\Omega,\mu)$ be a measure space and consider the $L^2$-space $L^2(\Omega,\mu)=\{f:\Omega\to\mathbb{C}:\int_\Omega|f|^2d\mu<\infty\}$. This is an infinite dimensional Hilbert space with the inner product $\langle f,g\rangle=\int_\Omega f(\omega)\overline{g(\omega)}d\mu(\omega)$. Let $L^\infty(\Omega,\mu)=\{f:\Omega\to\mathbb{C}:\exists\alpha<\infty,|f(\omega)|<\alpha~a.e.\}$ be the set of all measurable functions that are bounded almost everywhere. A multiplication operator $M_\varphi$ on $L^2(\Omega,\mu)$ is defined by 
\begin{align}
\begin{aligned}
    &L^\infty(\Omega)\ni\varphi\to M_\varphi\in B(L^2(\Omega,\mu)) \\
    &M_\varphi(f)=\varphi f
\end{aligned}
\end{align}
The set of all multiplication operators is a von Neumann algebra. Clearly they are fundamentally important in quantum mechanics, however it is an extremely non-trivial task to construct quantum gates to approximate them using only Pauli operators.

Next, let us explain the motivation for using von Neumann algebra from the viewpoint of algebraic quantum field theory. The physical observables are represented by self-conjugate operators. They may be non-bounded, but it is sufficient to consider only bounded operators by considering their spectral projections. In fact, all experimental data belong to some bounded set. We can consider observables in each bounded spacetime domain and the von Neumann algebras generated by their spectral projections. This yields a family of von Neumann algebras parametrized in the spacetime domain, which is called a net of von Neumann algebras. Moreover a von Neumann algebra is easily obtained by any subalgebra $\mathcal{B}$ of $B(\mathcal{H})$. Let $\mathcal{B}'$ be the commutant of $\mathcal{B}$ \eqref{eq:vonNeumann} and $\mathcal{B}''=(\mathcal{B}')'$. One can show that $\mathcal{B}\subset\mathcal{B}''$ and $\mathcal{B}'''=\mathcal{B}'$. Therefore $\mathcal{B}'$ is a von Neumann algebra based on the double commutant theorem.

In our theory, the interactions of quantum states among fibers can be understood as a quantum communication/interaction on the net of von Neumann algebras, which is a generalized quantum network on an arbitrary topological space. Therefore, from this perspective, our theory of computation with von Neumann algebras on am arbitrary topological space is a natural extension of conventional quantum computation in a form applicable to the algebraic quantum field theory.

\section{\label{sec:AQFT}Application to Algebraic Quantum Field Theory}
Now let us explain how we implement algebraic quantum field theory (AQFT) in our theory. To begin with, we give the general definition of AQFT, which consists of the following date:
\begin{itemize}
    \item{\textbf{Base space}:} The base space $X$ where AQFT is defined is a Minkowski space.
    \item{\textbf{Local algebra}:} On each bounded open set $O$ of $X$, a von Neumann algebra $\mathcal{A}(O)$ satisfying the following conditions is defined:
    \begin{enumerate}
        \item{Isotony:} $\mathcal{A}(O_1)\subset\mathcal{A}(O_2)$ for all bounded open sets $O_1,O_2$ such that $O_1\subset O_2$.
        \item{Causality:} For all bounded open sets $O_1,O_2$, when they are causally disjoint, then
        \begin{equation}
            A_1A_2-A_2A_1=0,
        \end{equation}
        for all $A_i\in\mathcal{A}(O_i)~(i=1,2)$.
        \item{Covariance:} Let $x\in X$, $\alpha\in SL(2,\mathbb{C})$, $\Lambda(\alpha)$ be a Lorentz matrix, and $U(x,\alpha)$ be a continuous unitary representation of the covering group of the Poincar\'{e} group. They obey the following equation
        \begin{equation}
            U(x,\alpha)\mathcal{A}(O)U(x,\alpha)^*=\mathcal{A}(\Lambda(\alpha)O+x).
        \end{equation}
    \end{enumerate}
\end{itemize}

The construction of a local algebra on a bounded open set $O$ can be done by collecting local algebras on $x\in O$. In particular, since condition (3) imposes a continuous deformation of the local $\mathcal{A}$, we may assume a sheaf of von Neumann algebras. For this we use the same technique by which we construct a sheaf by a presehaf. Let $r_{UV}:\mathcal{A}(U)\to\mathcal{A}(V)$ be the restriction map that realizes the isotony condition of local algebras which are defined on bounded open sets $V\subset U\subset X$. Furthermore imposing the following condition on the restriction map is consistent with the isotony condition:
\begin{equation}
   r_{VW}\circ r_{UV} =r_{UW}
\end{equation}
for bounded open sets $W\subset V\subset U$. In other words, AQFT can be regarded as a (pre)sheaf whose sections are von Neumann algebras. Then by taking the inductive limit
\begin{equation}
    \mathcal{A}_x=\lim_{U\to x}\mathcal{A}(U),
\end{equation}
we have a subset $\mathcal{A}_x$ of bounded operators. We may assume that this is a von Neumann algebra. Then for a bounded open set $O$ we can realize a local algebra that obeys the first constraint in such a way that
\begin{equation}
    \mathcal{A}(O)=\bigcup_{x\in O}\mathcal{A}_x.
\end{equation} 
The assignment $O\mapsto\mathcal{A}(O)$ is called a net of local algebras. By extending $O$ to the entire space $X$, one can construct a global algebra $\mathcal{A}(X)=\bigcup\mathcal{A}(O)$. Equivalently, one can construct a local algebra by restricting $\mathcal{A}(X)$ to each bounded open subset $O$ as $\mathcal{A}(O)=r_{XO}(\mathcal{A}(X))$. Note $\mathcal{A}(X)$ may not be a von Neumann algebra in general. 

Realization of the second and the third constrains of a local algebra can be done by implementing the Hamiltonian dynamics. Let $O\subset X$ be a bounded open set and $H(O)\in\text{End}(\mathcal{H})$ be a hermitian operator $(H(O)=H^*(O))$ defined on a Hilbert space $\mathcal{H}$ where $\mathcal{A}(O)$ is defined. When $H(O)$ is time-dependent $H=(O)=H^t(O)$, its time evolution can be defined by $\mathcal{U}_t=T_t\exp(-i\int^t_0H(t,O)dt)\in\text{End}(\mathcal{H})$, where $T_t\text{exp}$ is the time-ordered exponential. It can be approximated by a family $\{U_t\}\subset \mathcal{A}(O)$ of operators such that 
\begin{equation}
    \left\|T_t\exp(-i\int^t_0H^t(O)dt)-U_t\right\|<\epsilon,
\end{equation}
where $\epsilon>0$ is a precision of simulation. As long as the Hamiltonian $H^t(O)$ is designed to satisfy the laws of physics (particularly the conditions of relativity), locality conditions (2) and (3) should be automatically satisfied in the 0-limit of $\epsilon$. The quantum channel sending information of $\mathcal{A}(O_1)$ to different $\mathcal{A}(O_2)$ can be interpreted as propagation of particles. This can be discussed by considering time evolution of a state which is defined on an open set $O$ which contains $O_1\cup O_2\subset O$. Let us consider a time-evolution of an initial state $\rho_{0}$ defined on a bounded open set $O$. The state at time $t$ can be written as 
\begin{equation}
\rho_t=\mathcal{U}_t\rho_{0}\mathcal{U}^*_t.
\end{equation}
Then the measurement of a hermitian operator (physical observable) on $O$ can be done by means of a POVM operator. By restricting $\rho_t$ to $O_1$ or $O_2$, we obtain its local information.  

\section{Applications to Semiclassical Phenomena}
\subsection{\label{sec:semiclassical_operator}Semiclassical Operations}
To apply our theory to semi-classical and classical phenomena, we consider a semi-classical class of operators.

Let $f:\widetilde{D}(\mathcal{H})\to\mathbb{R}_{\ge0}$ be a non-negative map. Let 
\begin{equation}
    D_0(f)=\{\rho\in\widetilde{D}(\mathcal{H}):f(\rho)=0\}
\end{equation}
be the set of all elements of $\widetilde{D}(\mathcal{H})$ sent to 0 by $f$. For a given von Neumann algebra $\mathcal{A}$, we define its subset as 
\begin{equation}
    \mathcal{A}^f=\left\{a\in \mathcal{A}: \forall\rho\in \widetilde{D}(\mathcal{H}),~i_a(\rho)\in\widetilde{D}(\mathcal{H}),~f\left(\frac{i_a(\rho)}{\Tr(i_a(\rho))}\right)\le f\left(\frac{\rho}{\Tr\rho}\right)\right\}.
\end{equation}

We put 
\begin{equation}
\mathcal{C}\left(\mathcal{D}_0(f),\mathcal{A}^f,\mathcal{H}\right)=\bigcup_{\rho\in{D_0(f)}}O_{\mathcal{A}^f}(\rho)
\end{equation}
The following statement plays a fundamental role for discussing classical states and classical operators (Definition~\ref{def:semiclassical}).
\begin{theorem}
\label{thm:classical}
Let $\mathcal{A}$ be a von Neumann algebra. Then $\mathcal{C}\left(D_0(f),\mathcal{A}^f,\mathcal{H}\right)=D_0(f)$ is true for any non-negative $f:\widetilde{D}(\mathcal{H})\to\mathbb{R}_{\ge0}$.
\end{theorem}
\begin{proof}
It is obvious that $\mathcal{C}\left(D_0(f),\mathcal{A}^f,\mathcal{H}\right)\supset D_0(f)$. So we show $\mathcal{C}\left(D_0(f),\mathcal{A}^f,\mathcal{H}\right)\subset{D}_0(f)$. Suppose there is an element $\rho$ of $\mathcal{C}\left(D_0(f),\mathcal{A}^f,\mathcal{H}\right)$ that is not an element of $D_0(f)$. By definition, such a $\rho$ can be written as $\rho=\frac{i_a(\rho_0)}{\Tr(i_a(\rho_0))}$ with an element $a$ of $\mathcal{A}^f$ and a state $\rho_0$ in $D_0(f)$. They obey $0\le f(\rho)=f\left(\frac{i_a(\rho_0)}{\Tr(i_a(\rho_0))}\right)\le f\left(\frac{\rho_0}{\Tr\rho_0}\right)=0$, which means $\rho$ is also an element of $D_0(f)$. This contradicts the assumption that $\rho$ is not an element of $D_0(f)$. Therefore any element of $\mathcal{C}\left(D_0(f),\mathcal{A}^f,\mathcal{H}\right)$ is an element of $D_0(f)$.  
\end{proof}

The following statement can be shown in the same manner as Proposition~\ref{eq:selfadjoint}.
\begin{proposition}
For any non-negative $f$ and for any von Neumann algebra $\mathcal{A}$, the following set
\begin{equation}
    \left\{a\in \mathcal{A}: \forall\rho\in \widetilde{D}(\mathcal{H}),~i_a(\rho)\in\widetilde{D}(\mathcal{H}),~f\left(\frac{i_a(\rho)}{\Tr(i_a(\rho))}\right)= f\left(\frac{\rho}{\Tr\rho}\right)\right\}\cap\mathcal{R}
\end{equation}
is a group.
\end{proposition}

\begin{proposition}
For any non-negative $f$ and for any von Neumann algebra $\mathcal{A}$, the following set
\begin{equation}
    \left\{a\in \mathcal{A}: \forall\rho\in \widetilde{D}(\mathcal{H}),~i_a(\rho)\in\widetilde{D}(\mathcal{H}),~f\left(\frac{i_a(\rho)}{\Tr(i_a(\rho))}\right)< f\left(\frac{\rho}{\Tr\rho}\right)\right\}\cap\mathcal{R}
\end{equation}
is not a group.
\end{proposition}
\begin{proof}
The proof is completed by checking that the inverse $a^{-1}$ of $a$ is not contained in the set due to the relation 
    \begin{equation}
        f\left(\frac{i_a(\rho)}{\Tr(i_a(\rho))}\right)<f\left(\frac{\rho}{\Tr\rho}\right). 
    \end{equation}
\end{proof}

Let us explain below why a tuple $(f,\mathcal{A}^f,D_0(f))$ corresponds to a set of operators in the von Neumann algebra $\mathcal{A}$.
Whether a state is truly quantum or classical is determined by the presence or absence of quantum correlations. There are two known quantum correlations: entanglement and quantum discord.

In the classification of quantum states by quantum entanglement, all separable states are regarded as classical states, and the other states are regarded as true quantum states. The presence or absence of entanglement can be determined using an entanglement measure. Let us explain this more precisely. Let $\mathcal{D}_\text{sep}\subset \widetilde{D}(\mathcal{H})$ be the set of all separable states. Here we say $\rho\in\widetilde{D}(\mathcal{H})$ is separable if $\rho/\Tr\rho$ is a separable as an element of $D(\mathcal{H})$. Let $E:\text{End}(\mathcal{H})\to\text{End}(\mathcal{H})$ be an entanglement measure $E$, which obeys the following properties~\cite{PhysRevLett.78.2275}:
\begin{itemize}
    \item $E(\rho)\ge0$ for any density operator $\rho$. 
    \item $E(\rho)=0$ if and only if $\rho$ is separable.
    \item $E(\rho)$ is unchanged under any local unitary operation. 
    \item $E(\rho)$ does not increase by a local general measurement and classical communication. 
\end{itemize}
To avoid the appearance of entanglement in classical theory, we use the fourth property of an entanglement measure in the definition of semi-classical algebras.

\begin{example}
\label{ex:EEmeasure}
One of the most standard entanglement measures is defined using the quantum relative entropy $S(\rho\|\sigma)=\Tr(\rho\log\rho-\rho\log\sigma)$ $\rho,\sigma\in\mathcal{D}(\mathcal{H})$ \cite{umegaki1962conditional} as follows
\begin{equation}
    E(\rho)=\inf_{\sigma'\in\mathcal{D}_\text{sep}}S\left(\rho~\Bigg\|\frac{\sigma'}{\Tr\sigma'}\right),~\rho\in\mathcal{D}(\mathcal{H}).
\end{equation}
Besides, negativity and logarithmic negativity are also often used practically as entanglement measures.
\end{example}

Another measure of quantum correlation is quantum discord, which studies non-classical correlation between two subsystems. While non-separable states are called entangled states, even separate states may have non-zero value of discord. Let $X$ be a topological space. For a given state $\rho$ on $X$, suppose quantum states are well-defined on a subset $A\subset X$ and its complement $A^c=X\setminus A$ so that we can take the partial trace over each subsystem $\rho_A=\Tr_{\mathcal{H}_{A^c}}\rho$, $\rho_{A^c}=\Tr_{\mathcal{H}_{A}}\rho$. We introduce mutual information $I(\rho)=S(\rho_A)+S(\rho_{A^c})-S(\rho)$ and $J_A(\rho)=S(\rho_A)-S_{A|A^c}$, where $S_{A|A^c}$ is the conditional entropy. Then quantum discord is defined as 
\begin{equation}
    \mathcal{D}_A(\rho)=I(\rho)-J_A(\rho),
\end{equation}
which is non-negative~\cite{bera2017quantum}. When using quantum discord to discucss quantum correlations, an element $\rho$ of the set for which $\mathcal{D}_A(\rho)$ is 0 is called a classical state.

\begin{definition}
\label{def:semiclassical}
Let $f$ be a measure of quantum correlation. For a given von Neumann algebra $\mathcal{A}$, we call an element of $\mathcal{A}^f$ a semi-classical operator and call an element of $D_0(f)$ a classical state. 
\end{definition}
To classify the computational power of semi-classical algebras, we introduce the following definition.

\begin{definition}
\label{def:classical_universal}
Let $f$ be a measure of quantum correlation. A semi-classical algebra $\mathcal{A}^f$ is called \textbf{universal} (in the classical sense) if it acts transitively on $D_0(f)$, namely 
\begin{equation}
    O_{\mathcal{A}^f}(\rho)=D_0(f)
\end{equation}
holds for any element $\rho$ of $D_0(f)$.
\end{definition}

When $f:\widetilde{\mathcal{D}}(\mathcal{H})\to\mathbb{R}_{\ge0}$ is a measure of quantum correlation, Theorem~\ref{thm:classical} guarantees that actions of any semi-classical algebra to the set of all classical states are closed. Hence it warrants Definition~\ref{def:semiclassical} and Definition~\ref{def:classical_universal}.

\begin{remark}
Note that the strength of quantum correlations is not directly related to the power of quantum computation. In fact, a Clifford gate set can increase entanglement, but it is not a universal gate set. 
\end{remark}

\subsection{\label{sec:homotopy}Homotopy}
Let $\{\rho_t\}_{t\ge0}$ be a sequence of quantum states with a fixed initial state $\rho_0$. Namely there exists a sequence $\{a_t\}_{t\ge0}$ of unitary operators such that $\rho_t=a_t\rho_0 a^*_t$. Let $\alpha:\{\rho_t\}_{t\ge0}\to\mathbb{R}$ be a monotonically non-increasing with respect to $t$: $\alpha(\rho_t)\ge \alpha(\rho_{t'})$ for all $t<t'$. Functions $\alpha$ with this property are widely used in physics. Specifically, the measure of quantum correlations, temperature, energy, entropy could also be used as this function. In other words, in a situation where there is flow in a certain direction, some physical quantity serves as such an $\alpha$.

More generally, let $\alpha:X\to\mathbb{R}$ be a map, $f:[0,1]\to X$ be a continuous map and $\widetilde{f}_\alpha(t):[0,1]\to X$ be a map defined as

\begin{equation}
\widetilde{f}_\alpha=f~\text{if}~\alpha(f(t))~\text{is a non-increasing function of}~t,\text{otherwise}~\widetilde{f}_\alpha([0,1])=\emptyset.
\end{equation}

\begin{example}
    If we choose as $f_t$ a quantum state $\rho_t$ at time $t$, $X$ a certain set of quantum states, and $\alpha$ the entanglement measure, then the non-increasing property of $\alpha$ correspond to the example discussed in Sec.~\ref{sec:semiclassical_operator}.
\end{example}

For a given map $\alpha:X\to\mathbb{R}$ and given two paths $f,g$, which satisfy $f(1)=g(0)$, we define a product $\widetilde{f}_\alpha*\widetilde{g}_\alpha:[0,1]\to X$ as 
\begin{equation}
\widetilde{f}_\alpha*\widetilde{g}_\alpha=f*g~\text{if}~\alpha(f*g(t))~\text{is a non-increasing function of}~t,\text{otherwise}~\widetilde{f}_\alpha*\widetilde{g}_\alpha([0,1])=\emptyset.
\end{equation}

Let $f_0,f_1:[0,1]\to X$ be two continuous paths such that $f_0(0)=f_1(0)$ and $f_0(1)=f_1(1)$. Let $\alpha:X\to\mathbb{R}$ be a map. Suppose $f_0,f_1$ are homotopic by a homotopy $F:[0,1]\times [0,1]\to X$. We define $\widetilde{F}_a(s,t)$ as
\begin{equation}
\widetilde{F}_a=F~\text{if for each}~s,\alpha(F(s,t))~\text{is a non-increasing function of}~t,\text{otherwise}~\widetilde{F}_\alpha([0,1]\times[0,1])=\emptyset.
\end{equation}

For a given $\alpha:[0,1]\to\mathbb{R}$ and $f_0,f_1:X\to Y$ continuous maps if there exists a homotopy $F$ between $f_0$ and $f_1$ such that
\begin{align}
\begin{aligned}
    \widetilde{F}_\alpha(s,0)=F(s,0)&=f_0(s)\\
    \widetilde{F}_\alpha(s,1)=F(s,1)&=f_1(s)\\
    \widetilde{F}_\alpha(0,t)=F(0,t)&=f_0(0)\\
    \widetilde{F}_\alpha(1,t)=F(1,t)&=f_1(1)
\end{aligned}
\end{align}
then we call $\widetilde{F}_\alpha$ an $\alpha$-homotopy from $f_0$ to $f_1$, and write 
\begin{equation}
    f_0{\stackrel{\sim}{\to}}_\alpha f_1
\end{equation}

When they satisfy 
\begin{equation}
    f_0{\stackrel{\sim}{\to}}_\alpha f_1~\text{and}~~f_1{\stackrel{\sim}{\to}}_\alpha f_0,
\end{equation}
then we write $f_0\simeq_\alpha f_1$.
\begin{proposition}
\label{prop:const}
Suppose $f_0{\stackrel{\sim}{\to}}_\alpha f_1$ is satisfied. Then $f_0\simeq_\alpha f_1$ is true if and only if $\alpha(F(s,t))$ is constant with respect to $t$ for any homotopy $F$ between $f_0$ and $f_1$. 
\end{proposition}
\begin{proof}
    If $\alpha$ is constant, it is clear that $f_1{\stackrel{\sim}{\to}}_\alpha f_0$. 
    
    Suppose $\alpha(F(s,t))$ is not constant. We show that there is no $\alpha$-homotopy $\widetilde{H}_\alpha$ from $f_1$ to $f_0$. For this, we show that assuming the existence of such an $\alpha$-homotopy leads to a contradiction. Let $\widetilde{H}_{\alpha}(s,t)$ be an $\alpha$-homotopy from $f_1$ to $f_0$. Since $\alpha(H(s,t))$ is an non-increasing function of $t$, we have $\alpha(f_1(s))\ge \alpha(H(s,t))\ge\alpha(f_0(s))$ for all $s\in [0,1]$ and $t\in[0,1]$. Let $\widetilde{F}_\alpha$ be an $\alpha$-homotopy from $f_0$ to $f_1$. Since $\alpha(F(s,t))$ is an non-increasing function of $t$, we have $\alpha(f_0(s))\ge \alpha(F(s,t))\ge\alpha(f_1(s))$ for all $s\in [0,1]$ and $t\in[0,1]$. Moreover since both $H$ and $F$ are continuous, for any $(s,t)\in[0,1]\times[0,1]$, there exists $(s,t')\in[0,1]\times[0,1]$ such that $H(s,t)=F(s,t')$.

    Since we assume that $\alpha(F(s,t))$ is not constant, there is $(s_*,t_*)\in[0,1]\times[0,1]$ such that $\alpha(f_0(s_*))>\alpha(F(s_*,t_*))$. Then there also exists $t'_*\in[0,1]$ such that $F(s_*,t_*)=H(s_*,t'_*)$ and $\alpha(f_0(s_*))>\alpha(H(s_*,t'_*))\ge\alpha(f_0(s_*))$. However such $\alpha(H(s_*,t'_*))\in\mathbb{R}$ dose not exist. Therefore there is no $(s,t)\in[0,1]\times[0,1]$ such that $\alpha(f_0(s))>\alpha(F(s,t))$. This contradicts that $\alpha(F(s,t))$ is not constant.  
\end{proof}

Let $x_0\in X$ be a point and $\Omega(X,x_0)=\{f:[0,1]\to X,f~\text{is continuous},~f(0)=f(1)=x_0\}$ be the loop space with basepoint $x_0$. 
\begin{proposition}
    The relation $\simeq_\alpha$ is an equivalence relation on $\Omega(X,x_0)$
\end{proposition}
\begin{proof}
The reflexivity is clear. To show the symmetry relation, suppose $f_0\simeq_\alpha f_1$. Let $\widetilde{F}_\alpha(s,t)$ be such an $\alpha$-homotopy from $f_0$ and $f_1$. Due to Proposition~\ref{prop:const}, $\alpha(F(s,t))$ is constant. Then by puttting $H(s,t)=F(s,1-t)$, we find $\tilde{H}_{\alpha}$ is an $\alpha$-homotopy from $f_1$ to $f_0$.  

To show the transitivity, we assume $f_0\simeq_\alpha f_1$ and $f_1\simeq_\alpha f_2$. Let $\widetilde{F}_\alpha,\widetilde{H}_\alpha$ be their $\alpha$-homotopy. Then, we can create an $\alpha$-homotopy from $f_0$ to $f_2$ by
\begin{equation}
    \widetilde{G}_\alpha(s,t)=
    \begin{cases}
        \widetilde{F}_\alpha(s,2t)& 0\le t\le 1/2\\
        \widetilde{H}_\alpha(s,2t-1)& 1/2\le t\le 1
    \end{cases}
\end{equation}
\end{proof}

\begin{definition}
    We write $\pi_1(X,x_0,\alpha)$ for the quotient set of $\Omega(X,x_0)$ by $\simeq_\alpha$, and write $[f]_{(\alpha,x_0)}$ for an equivalent class of $f\in \Omega(X,x_0)$ by $\simeq_\alpha$. 
\end{definition}

\begin{corollary}
    $\pi_1(X,x_0,\alpha)$ is a group with the loop product. 
\end{corollary}

Two continuous maps $f:X\to Y$ and $g:Y\to X$ are called $\alpha$-homotopy equivalence if there are $\alpha:X\to\mathbb{R}$ and $\beta:Y\to\mathbb{R}$ such that
\begin{align}
\begin{aligned}
    g\circ f&\simeq_\alpha id_X\\
    f\circ g&\simeq_\alpha id_Y.
\end{aligned}
\end{align}
If such $f$ and $g$ exist for $X$ and $Y$, we call $X$ and $Y$ are $\alpha$-homotopy equivalent ($X\simeq_\alpha Y$). 

\begin{proposition}
\label{prop:alphahomotopy}
    $\simeq_\alpha$ is an equivalence relation of topological spaces. 
\end{proposition}
\begin{proof}
    The reflexive relation and the symmetric relation are trivial. We show the transitive relation ($X\simeq_\alpha Y$, $Y\simeq_\alpha Z\Rightarrow X\simeq_\alpha Z$.) By assumption, there are maps $f:X\to Y$, $f':Y\to Z$, $g:Y\to X$ and $g':Z\to Y$ such that $f\circ g\simeq_\alpha 1_Y$, $g\circ f\simeq_\alpha 1_X$, $f'\circ g'\simeq_\alpha 1_Z$ and $f'\circ g'\simeq_\alpha 1_Y$. By putting $f''=f'\circ f$ and $g''=g\circ g'$, it is strightforward to see
    \begin{align}
    \begin{aligned}
        f''\circ g''&=f'\circ(f\circ g)\circ g'\simeq_\alpha f'\circ 1_Y\circ g'=1_Z\\
        g''\circ f''&=g\circ(g'\circ f')\circ f\simeq_\alpha g'\circ 1_X\circ f=1_X.
    \end{aligned}
    \end{align}
\end{proof}
\begin{remark}
The classification based on $\alpha$-homotopy precludes topological spaces from being identified by a unphysical process. For example, when the energy of a topological space is used as $\alpha$, only those that can be mapped to each other while keeping their energy constant are considered equivalent. When attempting to deform a real object, physical or chemical parameters should be considered. 

One simple example of this is the difference in chemical properties due to the difference in structure between the cis (Fig.~\ref{fig:cis}~left) and trans (Fig.~\ref{fig:cis}~right) isomers. Since the two molecular formulas are exactly the same, differing only in the point of bonding, it would seem that they could be mapped to each other if one could move electrons continuously, but this is not quite possible in practice. In fact, the carbon-carbon double bond is so strong that it cannot rotate at room temperature, and it takes a certain amount of energy to make this possible. In general, the cis isomer is more energetically unstable than the trans isomer, so the transition from the cis isomer to the trans isomer can occur easily (and in some cases spontaneously), but the reverse is not true. 

\begin{figure}[H]
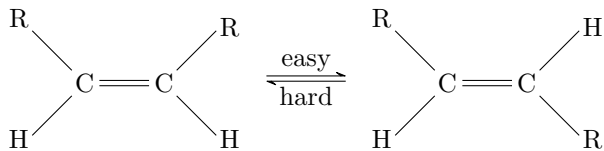

\setchemfig{scheme debug=false}
\begin{center}
\schemestart
\chemfig{R| -[7] C(-[5] \llap{H}) = C(-[7] {}H) -[1] {}R}
\arrow{<=>[\text{easy}][\text{hard}]}
\chemfig{R| -[7] C(-[5] \llap{H}) = C(-[7] {}R) -[1] {}H}
\schemestop
\caption{\label{fig:cis}Cis and trans isomers.}
\end{center}
\end{figure}
\end{remark}

\section{\label{sec:chemi}Algebraic Geometric Approach to Quantum Chemistry and Polymer Chemistry Based on Quantum Computation}
\subsection{Preliminaries}
With the advent and development of quantum computers, quantum chemistry is entering an extremely exciting era. While substances with high molecular weight behave in a classical mechanical manner, substances with small molecular weight exhibit a pronounced quantum mechanical behavior. It is meaningful to be able to simulate chemical properties using large-scale quantum computers in the near future. In addition to that, quantum chemistry will be of great interest to us to better understand theories of physics as an interdisciplinary field of Newtonian physics, quantum physics, statistical mechanics, and quantum information. So, within the scope of our theory, let us discuss reactions between molecules in terms of algebraic geometrical quantum calculations.

Below we comment on the theoretical framework. The following discussion in this section focuses on conceptual arguments at the expense of rigor. When defining quantum fibrations (Def.~\ref{def:fibration}), we did not assume homotopy equivalence between fibers. The implications of this will become clearer by examining more specific examples. There are many substances with different local physical and chemical properties, and local structures can be changed by chemical reactions.

\subsection{Quantum Model of Polymers}
A polymer is a molecule with high molecular weight that is composed of repeated linked units. The properties of polymers are determined by the chemical structure of the basic units and how they are geometrically linked to each other. Even when a polymer is composed of one type of monomer, various structures appear in a chain due to the different bonding of each monomer. Polymers consisting of two or more repeating units are called copolymers and are classified according to the way each unit is bonded. Polymers in which multiple units appear in succession are called block copolymers (Fig.\ref{fig:polymer} [Upper]), those in which the units are joined alternately are called alternating copolymers (Fig.\ref{fig:polymer} [Middle]), and those in which the units are joined randomly are called random copolymers (Fig.\ref{fig:polymer} [Lower]). In addition, there are some branched polymers and ring polymers. Furthermore, there are polymers that lose some of their regularity and have a steric structure.

\begin{figure}[H]
\begin{minipage}{\textwidth}
\centering
    \begin{tikzpicture}[node1/.style={draw=white, fill = orange, circle, minimum size=0.5cm}, node2/.style={draw=white, fill = blue!60, circle, minimum size=0.8cm}]
      \node[node1, line width=0.1mm] (s1) {};
      \node[node2, right = 0.25cm of s1, line width=0.1mm] (s2) {};
      \node[node1, left  = 0.25cm and 0.25cm of s1, line width=0.1mm] (s3) {};
      \node[node2, right  = 0.25cm and 0.25cm of s2, line width=0.1mm] (s4) {};
      \node[node1, left  = 0.2cm and 0.25cm of s3, line width=0.1mm] (s5) {};
      \node[node2, right  = 0.25cm and 0.25cm of s4, line width=0.1mm] (s6) {};
      \node[node1, left  = 0.25cm and 0.25cm of s5, line width=0.1mm] (s7) {};
      \node[node2, right  = 0.25cm and 0.25cm of s6, line width=0.1mm] (s8){};
      \node[node1, left  = 0.25cm and 0.25cm of s7, line width=0.1mm] (s9){};
      \node[node2, right  = 0.25cm and 0.25cm of s8, line width=0.1mm] (s10){};
        \path[-,thick, >=stealth]
        (s7) edge[below,orange,line width=1mm]  (s9)
        (s7) edge[below,orange,line width=1mm]  (s5)
        (s5) edge[below,orange,line width=1mm]  (s3)
        (s3) edge[below,orange, line width=1mm] (s1)
        (s1) edge[below,magenta, line width=1mm]  (s2)
        (s2) edge[below,blue!60,line width=1mm]  (s4)
        (s4) edge[below,blue!60,line width=1mm]  (s6)
        (s6) edge[below,blue!60,line width=1mm]  (s8)
        (s8) edge[below,blue!60,line width=1mm]  (s10);
    \end{tikzpicture}
\vspace{0.5cm}
\end{minipage}
\vspace{0.5cm}
\begin{minipage}{\textwidth}
\centering
    \begin{tikzpicture}[node1/.style={draw=white, fill = orange, circle, minimum size=0.5cm}, node2/.style={draw=white, fill = blue!60, circle, minimum size=0.8cm}]
      \node[node1, line width=0.1mm] (s1) {};
      \node[node2, right = 0.25cm of s1, line width=0.1mm] (s2) {};
      \node[node2, left  = 0.25cm and 0.25cm of s1, line width=0.1mm] (s3) {};
      \node[node1, right  = 0.25cm and 0.25cm of s2, line width=0.1mm] (s4) {};
      \node[node1, left  = 0.2cm and 0.25cm of s3, line width=0.1mm] (s5) {};
      \node[node2, right  = 0.25cm and 0.25cm of s4, line width=0.1mm] (s6) {};
      \node[node2, left  = 0.25cm and 0.25cm of s5, line width=0.1mm] (s7) {};
      \node[node1, right = 0.25cm and 0.25cm of s6, line width=0.1mm] (s8){};
      \node[node1, left = 0.25cm and 0.25cm of s7, line width=0.1mm] (s9){};
      \node[node2, right = 0.25cm and 0.25cm of s8, line width=0.1mm](s10){};
        \path[-,thick, >=stealth]
        (s7) edge[below,magenta, line width=1mm]  (s5)
        (s5) edge[below,magenta, line width=1mm]  (s3)
        (s3) edge[below,magenta, line width=1mm] (s1)
        (s1) edge[below,magenta, line width=1mm]  (s2)
        (s2) edge[below,magenta, line width=1mm]  (s4)
        (s4) edge[below,magenta, line width=1mm]  (s6)
        (s6) edge[below,magenta, line width=1mm]  (s8)
        (s7) edge[below,magenta, line width=1mm]  (s9)
        (s8) edge[below,magenta, line width=1mm]  (s10);
    \end{tikzpicture}
\end{minipage}
\begin{minipage}{\textwidth}
\centering
\begin{tikzpicture}[node1/.style={draw=white, fill = orange, circle, minimum size=0.5cm}, node2/.style={draw=white, fill = blue!60, circle, minimum size=0.8cm}]
      \node[node2, line width=0.1mm] (s1) {};
      \node[node2, right = 0.25cm of s1, line width=0.1mm] (s2) {};
      \node[node1, left  = 0.25cm and 0.25cm of s1, line width=0.1mm] (s3) {};
      \node[node1, right  = 0.25cm and 0.25cm of s2, line width=0.1mm] (s4) {};
      \node[node2, left  = 0.25cm and 0.25cm of s3, line width=0.1mm] (s5) {};
      \node[node2, right  = 0.25cm and 0.25cm of s4, line width=0.1mm] (s6) {};
      \node[node1, left  = 0.25cm and 0.25cm of s5, line width=0.1mm] (s7) {};
      \node[node1, right = 0.25cm and 0.25cm of s6, line width=0.1mm] (s8){};
     \node[node1, left = 0.25cm and 0.25cm of s7, line width=0.1mm] (s9){};
     \node[node2, right = 0.25cm and 0.25cm of s8, line width=0.1mm] (s10){};
        \path[-,thick, >=stealth]
        (s7) edge[below,magenta, line width=1mm]  (s5)
        (s5) edge[below,magenta, line width=1mm]  (s3)
        (s3) edge[below,magenta, line width=1mm] (s1)
        (s1) edge[below,blue!60, line width=1mm]  (s2)
        (s2) edge[below,magenta, line width=1mm]  (s4)
        (s4) edge[below,magenta, line width=1mm]  (s6)
        (s6) edge[below,magenta, line width=1mm]  (s8)
        (s7) edge[below,orange, line width=1mm]  (s9)
        (s8) edge[below,magenta, line width=1mm]  (s10);
    \end{tikzpicture}
\end{minipage}
\caption{[Upper] Block copolymer [Middle] Alternating copolymer [Lower] Random copolymer}
\label{fig:polymer}
\end{figure}
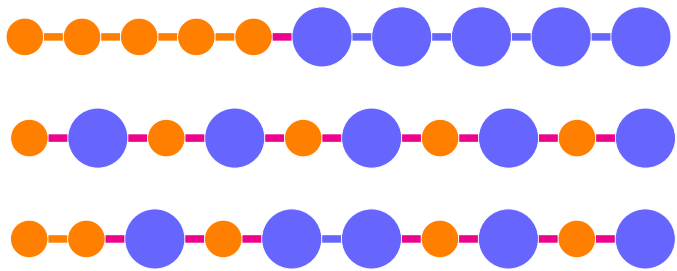

Determining the electronic structure of polymers is of various importance in terms of applications. For example, it can reveal the electrical properties of polymer surfaces, such as their susceptibility to electrification. It will also lead to the elucidation of the physical and chemical properties of proteins, opening up important prospects in the fields of life sciences, medicine, and drug discovery. Furthermore, highly electrically conductive polymer compounds, called semiconducting polymers, have potential in engineering applications for organic field-effect transistors and organic thin-film solar cells.

Regarding the base space, since those polymers are roughly one-dimensional chains, they all appear to be homotopic to a single point in the conventional homotopy theory, but given their energy structures (see also Fig.~\ref{fig:evol}), it would be difficult to deform them continuously to a single point (Prop.~\ref{prop:alphahomotopy}):
\begin{equation}
    \text{Block copolymer}\not\simeq_\alpha\text{Alternating 
 copolymer}\not\simeq_\alpha\text{Random copolymer}
\end{equation}
The structure of the base space and the structure of the fibers are related to each other since the shape of the base space is attributed to chemical bonds of the electrons and molecules associated with it.

We consider a model in which the Hilbert space may be different for each neighborhood. Let $X$ be a base space covered by finitely many non-overlapping open sets $X=\bigcup_{i}U_i,~U_i\cap U_j=\emptyset~(i\neq j)$ and the Hilbert space $\mathcal{H}_i$ of each neighbor $U_i$ is $\mathbb{C}^{n_i}$, where $n_i$ is a positive integer. Such a model can be easily implemented using a large scale universal quantum computer, since it can be represented by $n_i$ spin-1/2 particles coupled to site $i$. 

As a simple but practical example, we consider a polymer which consists of two basic units $A$ and $B$. Let $\text{label}=\{1,\cdots,N\}$ be the set of all labels of coordinates. We divide it into the two ordered set $\text{label}_A=\{a_1,\cdots,a_{n_A}\}$ of all labels of $A$ and that  $\text{label}_B=\{b_1,\cdots,b_{n_B}\}$ of $B$. So they obey $N=n_AN_A+n_BN_B$, $\text{label}=\text{label}_A\cup\text{label}_B$ and $\text{label}_A\cap\text{label}_B=\emptyset$. We further divide $\text{label}_A$ into two ordered sets $\text{link}_{AA}(\subset \text{label}_A)$ and $\text{link}_{AB}(\subset\text{label}_A)$ that accommodate information of links between the neighboring sites, and $\text{label}_B$ into two ordered sets $\text{link}_{BB}(\subset\text{label}_B)$ and $\text{link}_{BA}(\subset\text{label}_B)$ containing date of links between two neighboring sits as follows:
\begin{align}
\begin{aligned}
    a_i\in\text{link}_{AA}&\Longleftrightarrow a_i, a_i+1\in\text{label}_{A}\\
    a_i\in\text{link}_{AB}&\Longleftrightarrow a_i\in\text{label}_A, a_i+1\in\text{label}_{B}\\
    b_i\in\text{link}_{BB}&\Longleftrightarrow b_i, b_i+1\in\text{label}_{B}\\
    b_i\in\text{link}_{BA}&\Longleftrightarrow b_i\in\text{label}_B, b_i+1\in\text{label}_{A}
\end{aligned}
\end{align}
Note that $\text{link}_{AA}$ and $\text{link}_{AB}$, and $\text{link}_{BB}$ and $\text{link}_{BA}$ have overlaps, respectively. 

Then a generic Hamiltonian formalism of this polymer can be constructed as 
\begin{align}
\begin{aligned}
    H_\text{polymer}=&\sum_{a_i\in\text{label}_{AA}} \left(J^{AA}_{a_i}\widetilde{H}^A_{a_i}\widetilde{H}^A_{a_i+1}+h.c.\right)+\sum_{a_i\in\text{label}_{AB}}\left(J^{AB}_{a_i}\widetilde{H}^A_{a_i}\widetilde{H}^B_{a_i+1}+h.c.\right)\\
    &+\sum_{b_i\in\text{label}_{BB}}\left(J^{BB}_{b_i}\widetilde{H}^B_{b_i}\widetilde{H}^B_{b_i+1}+h.c.\right)+\sum_{b_i\in\text{label}_{BA}}\left(J^{BA}_{b_i}\widetilde{H}^B_{b_i}\widetilde{H}^A_{b_i+1}+h.c.\right)\\
    &+\sum_{a_i\in\text{label}_{A}}h^A_{a_i}H^A_{a_i}+\sum_{b_i\in\text{label}_{B}}h^B_{b_i}H^B_{b_i},
\end{aligned}
\end{align}
where $J^{AB}_{a_i}$ gives a coupling between two units $A$ at $a_i$ and $B$ at $a_{i}+1$, $H^A_{a_i}$ is a Hamiltonian acting on the Hilbert space of $A$ at $a_i$ and $\widetilde{H}^A_{a_i}$ is an operator of interacting term, for example. $H^A_{a_i}$ and $\widetilde{H}^A_{a_i}$ should have the same dimension. When the dimensions of the Hilbert spaces of $A$ and $B$ are $N_A$ and $N_B$, respectively, and do not depend on their locations, then the Hilbert space of this entire system is $\mathcal{H}_{tot}=\mathbb{C}^{n^{N_A}_A\times n^{N_B}_B}$. 

\subsection{Quantum Computational Chemistry}
We consider a tight-binding model consisting of the nearest-neighbor interactions, in which the type and number of coupled particles varies from bond to bond. The ground state of this Hamiltonian can be obtained by quantum annealing~\cite{1998PhRvE..58.5355K}
\begin{align}
\begin{aligned}
    H(t)&=f(t)H_\text{polymer}+(1-f(t))H_0
\end{aligned}
\end{align}
where $H_0=-\sum_{i=1}^{N}X_i$ is the sum of the Pauli $X$ operators at the $i$th site and $f(t)$ is a monotonically increasing continuous function from 0 to 1 with fixed initial $f(0)=0$ and end $f(1)=1$. 

The time evolution of the system is given by $U(t)=\exp\left(-i\int_0^tH(t)dt\right)$. We can construct a quantum fibration $\mathscr{F}$ on $X$ as follows. The initial state is the ground state $\ket{\psi}_{in}=\prod_{i}\ket{+}_i$ of $H_0$, where $\ket{+}_i$ is an eigenstate of $X_i$: $X_i\ket{+}_i=\ket{+}_i$. So the time evolution of the density operator of the entire system is 
\begin{equation}
    \rho(t)=U(t)\ket{\psi}_{in}\bra{\psi}_{in}U^*(t). 
\end{equation}
By restricting $\rho(t)$ on each $U_i$, we obtain a family of quantum states $\{\rho(t)|_{U_i}\}_{0\le t\le1}$ on $U_i$, which is a quantum orbit on $U_i$. Note that the quantum state at initial time is a pure state with no quantum correlations and interactions, and the local algebra is completely determined by local properties only. As the calculation progresses, the interaction between particles becomes non-negligible as the value of $f(t)$ increases from 0, and the effect of non-local interactions appears in the local algebra. 

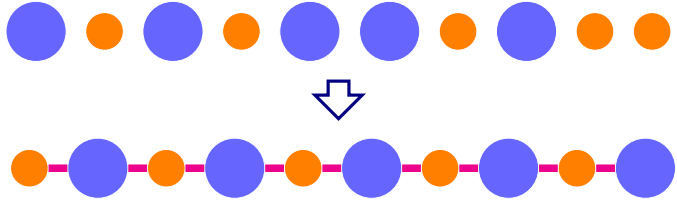
\begin{figure}[H]
\begin{minipage}{\textwidth}
\centering
    \begin{tikzpicture}[node1/.style={draw=white, fill = orange, circle, minimum size=0.5cm}, node2/.style={draw=white, fill = blue!60, circle, minimum size=0.8cm}]
      \node[node2, line width=0.1mm] (s1) {};
      \node[node2, right = 0.25cm of s1, line width=0.1mm] (s2) {};
      \node[node1, left  = 0.25cm and 0.25cm of s1, line width=0.1mm] (s3) {};
      \node[node1, right  = 0.25cm and 0.25cm of s2, line width=0.1mm] (s4) {};
      \node[node2, left  = 0.2cm and 0.25cm of s3, line width=0.1mm] (s5) {};
      \node[node2, right  = 0.25cm and 0.25cm of s4, line width=0.1mm] (s6) {};
      \node[node1, left  = 0.25cm and 0.25cm of s5, line width=0.1mm] (s7) {};
      \node[node1, right  = 0.25cm and 0.25cm of s6, line width=0.1mm] (s8){};
      \node[node2, left  = 0.25cm and 0.25cm of s7, line width=0.1mm] (s9){};
      \node[node1, right  = 0.25cm and 0.25cm of s8, line width=0.1mm] (s10){};
    \end{tikzpicture}
\vspace{0.2cm}
\end{minipage}
\begin{minipage}{\textwidth}
\centering
\begin{tikzpicture}%[>={Triangle[width=15mm,length=10mm]},->]
%  \draw[blue!50!black,line width=5mm] (0,1.5) -- (0,0);
\node[single arrow, draw=blue!50!black, fill=none, 
      minimum width = 3pt, single arrow head extend=5pt,
      minimum height=5mm, very thick, rotate=270] {}; % length of arrow
\end{tikzpicture}
\vspace{0.2cm}
\end{minipage}
\begin{minipage}{\textwidth}
\centering
    \begin{tikzpicture}[node1/.style={draw=white, fill = orange, circle, minimum size=0.5cm}, node2/.style={draw=white, fill = blue!60, circle, minimum size=0.8cm}]
      \node[node1, line width=0.1mm] (s1) {};
      \node[node2, right = 0.25cm of s1, line width=0.1mm] (s2) {};
      \node[node2, left  = 0.25cm and 0.25cm of s1, line width=0.1mm] (s3) {};
      \node[node1, right  = 0.25cm and 0.25cm of s2, line width=0.1mm] (s4) {};
      \node[node1, left  = 0.2cm and 0.25cm of s3, line width=0.1mm] (s5) {};
      \node[node2, right  = 0.25cm and 0.25cm of s4, line width=0.1mm] (s6) {};
      \node[node2, left  = 0.25cm and 0.25cm of s5, line width=0.1mm] (s7) {};
      \node[node1, right = 0.25cm and 0.25cm of s6, line width=0.1mm] (s8){};
      \node[node1, left = 0.25cm and 0.25cm of s7, line width=0.1mm] (s9){};
      \node[node2, right = 0.25cm and 0.25cm of s8, line width=0.1mm](s10){};
        \path[-,thick, >=stealth]
        (s7) edge[below,magenta, line width=1mm]  (s5)
        (s5) edge[below,magenta, line width=1mm]  (s3)
        (s3) edge[below,magenta, line width=1mm] (s1)
        (s1) edge[below,magenta, line width=1mm]  (s2)
        (s2) edge[below,magenta, line width=1mm]  (s4)
        (s4) edge[below,magenta, line width=1mm]  (s6)
        (s6) edge[below,magenta, line width=1mm]  (s8)
        (s7) edge[below,magenta, line width=1mm]  (s9)
        (s8) edge[below,magenta, line width=1mm]  (s10);
    \end{tikzpicture}
\end{minipage}
\caption{Time evolution from non-interacting molecules to alternating copolymer.}
\label{fig:altcopolymer}
\end{figure}

Now we consider a transition from an initial state in which there is no interaction between particles to a state in which there is interaction between particles (Fig.~\ref{fig:altcopolymer}). The initial Hamiltonian is 
\begin{align}
\begin{aligned}
     H_0&=\sum_{i=1}^{N}H_i,~~H_i=
     \begin{cases}
     H^A_{i}&i\in\text{label}_{A}\\
     H^B_{i}&i\in\text{label}_{A}
     \end{cases}
\end{aligned}
\end{align} 

Let $\mathcal{M}=\left\{\{(J_i,h_i)\}_{i=1}^{N}: J_{i}\in\{0,1,J^{AA}_i,J^{AB}_i,J^{BA}_i,J^{BB}_i\}, h_i\in\{0,1,h^A_i,h^B_i\}\right\}$ be the set of all families of all link parameters defined on $X$. Then a general Hamiltonian can be 
\begin{equation}
\label{eq:dynamics}
    H(t)=\sum_{i=1, C,D\in\{A,B\}}^{N-1}J_i(t)\left(\widetilde{H}^{C}_i\widetilde{H}^{D}_{i+1}+h.c.\right)+\sum_{i=1,C\in\{A,B\}}^Nh_i(t)\widetilde{H}^{C}_i
\end{equation}
where $J_i(t)$ is an element of $\{0,1,J^{AA}_i,J^{AB}_i,J^{BA}_i,J^{BB}_i\}$ and $h_i$ is an element of $\{0,1,h^A_i,h^B_i\}$ such that $J_i(0)=0$ for all $i\in\text{label}$ and $h_i(0)=h^A_i$ if $i\in\text{label}_A$ and $h_i(0)=h^B_i$ if $i\in\text{label}_B$. At the end of evolution $t=1$, we require $J_i(1)$ and $h_i(1)$ generate the configuration of the alternating copolymer. 

There is no creation or annihilation of particles during the entire process of time variation, and the total Hilbert space $\mathcal{H}_{tot}=\bigotimes_{i=1}^N\mathcal{H}_i$ on $X$ is identical, but interactions cause changes in the structure of the local Hilbert space as well as the local algebra:
\begin{align}
\begin{aligned}
\label{eq:transition}
    \mathscr{F}^{t=0}_{U_i}(\mathcal{H}_i)&\to \mathscr{F}^{t=1}_{U_i}(\mathcal{H}_{tot})\\
    \mathcal{A}^{t=0}_{U_i}(\mathcal{H}_i)&\to \mathcal{A}^{t=1}_{U_i}(\mathcal{H}_{tot}),
\end{aligned}
\end{align}
where $\mathcal{A}^{t=1}_{U_i}(\mathcal{H}_{tot})$ is a subalgebra of the algebra $\mathcal{A}^{t=1}_{X}(\mathcal{H}_{tot})$ including the effects of interactions. Note that at the initial time, the local algebra is completely determined by the local property since there are no interactions and correlations among particles. Even if the Hamiltonian consists only of nearest-neighbor interactions, after a sufficient amount of time, local information will spread throughout the system. Therefore, local quantities and local algebras are determined by global properties of the system. If the final state is energetically stable and close to a thermal equilibrium state, there is no return from the final state to the initial state since the local Hilbert space of the final state is not isomorphic to the local Hilbert space of the initial state. This embedding of a local initial structure into the global structure can be regarded as a one-way map. 

Figure~\ref{fig:evol} illustrates this time evolution conceptually. To have a chemical reaction, reactants must once be in a state that facilitates recombination between atoms (activated state), and the energy required to achieve this state is called the activation energy. Catalysts, which remain unchanged after a reaction, facilitate chemical reactions by lowering the activation energy. Therefore it plays a role of  a one-way map. For example, chiral catalysts make it easier to synthesize certain substances. In the middle of a reaction, the reactant weakly bonds with the catalyst to form another compound, and the reaction proceeds via this reaction intermediate.

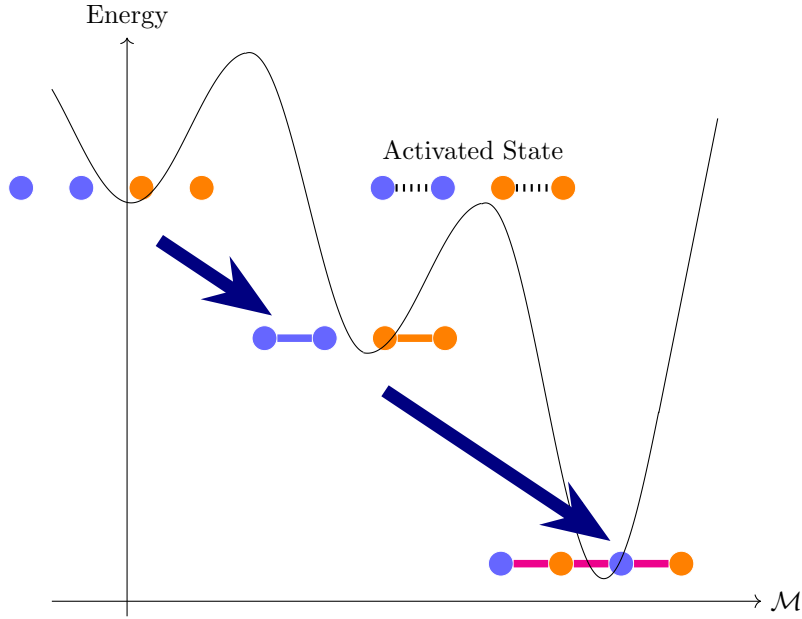
\begin{figure}[H]
    \centering
    \begin{tikzpicture}[node1/.style={draw=white, fill = orange, circle, minimum size=0.2cm}, node2/.style={draw=white, fill = blue!60, circle, minimum size=0.32cm}]
    \node[node2, line width=0.1mm] at (-1+0.37*pi,6.2){};
    %\node[node1, line width=0.1mm] at (-0.6+0.37*pi,6.2){};
    \node[node2, line width=0.1mm] at (-0.2+0.37*pi,6.2){};
    %\node[node1, line width=0.1mm] at (0.2+0.37*pi,6.2){};
    \node[node1, line width=0.1mm] at (0.6+0.37*pi,6.2){};
    %\node[node1, line width=0.1mm] at (1.0+0.37*pi,6.2){};
    \node[node1, line width=0.1mm] at (1.4+0.37*pi,6.2){};
    %\node[node2, line width=0.1mm] at (1.8+0.37*pi,6.2){};
    \node[node2, line width=0.1mm] at (1.4*pi-1,4.2)(s1){};
    %\node[node2, line width=0.1mm] at (1.4*pi-0.6,4.2)(s2){};
    \node[node2, line width=0.1mm] at (1.4*pi-0.2,4.2)(s3){};
    %\node[node2, line width=0.1mm] at (1.4*pi+0.2,4.2)(s4){};
    \node[node1, line width=0.1mm] at (1.4*pi+0.6,4.2)(s5){};
    %\node[node1, line width=0.1mm] at (1.4*pi+1,4.2)(s6){};
    \node[node1, line width=0.1mm] at (1.4*pi+1.4,4.2)(s7){};
    %\node[node1, line width=0.1mm] at (1.4*pi+1.8,4.2)(s8){};
    \path[-,thick, >=stealth]
        (s1) edge[below,blue!60, line width=1mm]  (s3)
        %(s2) edge[below,magenta, line width=1mm]  (s3)
        %(s3) edge[below,magenta, line width=1mm] (s5)
        %(s4) edge[below,magenta, line width=1mm]  (s5)
        (s5) edge[below,orange, line width=1mm]  (s7);
        %(s6) edge[below,magenta, line width=1mm]  (s7);
        %(s7) edge[below,magenta, line width=1mm]  (s8);
    \node[node2, line width=0.1mm] at (2.4*pi-1,1.2)(t1){};
    %\node[node1, line width=0.1mm] at (2.4*pi-0.6,1.2)(t2){};
    \node[node1, line width=0.1mm] at (2.4*pi-0.2,1.2)(t3){};
    %\node[node1, line width=0.1mm] at (2.4*pi+0.2,1.2)(t4){};
    \node[node2, line width=0.1mm] at (2.4*pi+0.6,1.2)(t5){};
    %\node[node1, line width=0.1mm] at (2.4*pi+1,1.2)(t6){};
    \node[node1, line width=0.1mm] at (2.4*pi+1.4,1.2)(t7){};
    %\node[node1, line width=0.1mm] at (2.4*pi+1.8,1.2)(t8){};
        \path[-,thick, >=stealth]
        (t1) edge[below,magenta,line width=1mm]  (t3)
        %(t2) edge[below,magenta,line width=1mm]  (t3)
        (t3) edge[below,magenta,line width=1mm] (t5)
        %(t4) edge[below,magenta,line width=1mm]  (t5)
        (t5) edge[below,magenta,line width=1mm]  (t7);
        %(t6) edge[below,magenta,line width=1mm]  (t7);
        %(t7) edge[below,magenta,line width=1mm]  (t8);
    \draw[very thin,color=gray] (-0.1,-1.1);
    \draw[->] (0.5*pi-1,0.7) -- (10,0.7) node[right] {$\mathcal{M}$};
    \draw[->] (0.5*pi,0.5) -- (0.5*pi,8.2) node[above] {Energy};
    \draw[color=black] node[right] {}; 
    \draw [domain=0.5*pi-1:pi, samples=100] plot (\x, {cos(2*\x r-2*pi)+7});
    \draw [domain=pi:1.5*pi, samples=100] plot (\x, {2*cos(2*\x r-2*pi)+6});
    \draw [domain=1.5*pi:2*pi, samples=100] plot (\x, {cos(2*\x r-2*pi)+5});
    \draw [domain=2*pi:2.75*pi, samples=100] plot (\x, {2.5*cos(2*\x r-2*pi)+3.5});
    \draw [domain=2.75*pi:3*pi, samples=100] plot (\x, {5*(\x-2.75*pi)+2.5*cos(2*2.75*pi-2*pi)+0.75});
    \draw [blue!50!black, -Stealth, line width=5pt] (2,5.5) -- (3.5,4.5);
    \draw [blue!50!black, -Stealth, line width=5pt] (5,3.5) -- (8,1.5);
    \node[node2, line width=0.1mm] at (-1+1.9*pi,6.2)(u1){};
    %\node[node1, line width=0.1mm] at (-0.6+0.37*pi,6.2){};
    \node[node2, line width=0.1mm] at (-0.2+1.9*pi,6.2)(u2){};
    %\node[node1, line width=0.1mm] at (0.2+0.37*pi,6.2){};
    \node[node1, line width=0.1mm] at (0.6+1.9*pi,6.2)(u3){};
    %\node[node1, line width=0.1mm] at (1.0+0.37*pi,6.2){};
    \node[node1, line width=0.1mm] at (1.4+1.9*pi,6.2)(u4){};
    %\node[node2, line width=0.1mm] at (1.8+0.37*pi,6.2){};
    \node at (0.2+1.9*pi,6.7) {Activated State};
    \path[-,thick, >=stealth,dotted]
        (u1) edge[below,black,line width=1mm]  (u2)
        %(t2) edge[below,magenta,line width=1mm]  (t3)
        (u3) edge[below,black,line width=1mm] (u4);
  \end{tikzpicture}
    \caption{Energy change from a non-interacting state to an interacting state. The system is determined by a family of coupling parameters $\{(J_i,h_i)\}_i\in\mathcal{M}$.}
    \label{fig:evol}
\end{figure}

\bibliographystyle{amsalpha}
\bibliography{ref}
\end{document}